\documentclass[a4paper,11pt,oneside,openany]{article}

\usepackage[utf8]{inputenc}
\usepackage{amssymb}
\usepackage{amsmath,amsthm}

\usepackage{mathrsfs}
\usepackage{xcolor}
\usepackage{latexsym}
\usepackage{epsfig}
\usepackage[margin=0.8in]{geometry}
\usepackage{authblk}
\usepackage[colorlinks=true, linkcolor=red, citecolor=blue, urlcolor=blue]{hyperref}

\usepackage[noend]{algpseudocode}
\usepackage[ruled,linesnumbered,noresetcount]{algorithm2e}

\usepackage{breakcites}
\usepackage{caption}
\usepackage{subcaption}
\usepackage{comment}
\usepackage{setspace}
\usepackage[numbers]{natbib}
\usepackage{bbm}
\usepackage{dsfont}
\usepackage{float}

\newtheorem{theorem}{Theorem}[section]

\newtheorem{lemma}[theorem]{Lemma}

\theoremstyle{definition}

\numberwithin{equation}{section}

\newenvironment{noindentsection}{\setlength{\parindent}{0pt}}{}

\begin{document}
\title{Adaptive Divide and Conquer with Two Rounds of Communication}

\author[1]{Niladri Kal}
\author[2]{Botond Szabó}
\author[1]{Rajarshi Guhaniyogi}
\author[3]{Natesh Pillai}
\author[4]{Debdeep Pati}

\affil[1]{Department of Statistics, Texas A\&M University, USA}
\affil[2]{Department of Decision Sciences, Bocconi University, Italy}
\affil[3]{Department of Statistics, Harvard University, USA}
\affil[4]{Department of Statistics, University of Wisconsin\textendash Madison, USA}

\date{}
\maketitle
\vspace{-2em}
\begin{abstract}
        \noindent
        We introduce a two-round adaptive communication strategy that enables rate-optimal estimation in the white noise model without requiring prior knowledge of the underlying smoothness. In the first round, local machines send summary statistics using $(\log_2(n))^2$ bits to enable the central machine to select the tuning parameters of the procedure. In the second round, another set of statistics are transmitted using optimal number of bits, enabling the central machine to aggregate and produce a final estimator that adapts to the true smoothness level. This approach achieves optimal convergence rates across a wider range of regularities, offering a potential improvement in the adaptability and efficiency of distributed estimation compared to existing one-round methods.
\end{abstract}

\section{Introduction}
In numerous recent applications, ranging from environmental science and forestry to imaging, large datasets have become ubiquitous. The sheer volume of these datasets often creates bottlenecks in analysis and storage when using a single machine. Moreover, many real-world scenarios involve data being partitioned across different organizations or locations, where strict security constraints limit the extent to which this data can be shared. Such considerations have resulted in a significant interest in the study of distributed statistical methods in the recent times \citep{Zhang2013Information-theoreticConstraints, Neiswanger2014EmbarrassingParallelMCMC, Kleiner2014AData, Deisenroth2015DistributedProcesses,Srivastava2015WASP:Posteriors, Zhang2015DivideRates, Scott2016BayesAlgorithm, Lee2017Communication-efficientRegression, Battey2018DistributedModels, Guhaniyogi2018Meta-kriging:Datasets,Guhaniyogi2019MultivariateKriging, Szabo2019AnMethods, Guhaniyogi2022DistributedPrior, Guhaniyogi2023DistributedData, Szabo2025AdaptationProcesses}. In a distributed setup, the data are divided across several machines. In each machine, the data is processed locally, and then the local results are suitably aggregated in a central machine to produce the final output. By leveraging de-centralized analysis in multiple machines, distributed systems facilitate parallel processing, which significantly reduces computation time, storage burden and allows for the efficient handling of large datasets as compared to traditional, centralized systems. Besides enhancing speed, distributed computing is also useful in enhancing security because of the decentralized structure of data storage and processing. However, designing distributed procedures effectively is challenging. One key challenge is to decide how to partition the data and how to combine the local results. If too few machines are used, the computational benefits of distributed setup are limited. Whereas, if the data are split across many machines, then each machine receives only a small portion of the data, leading to high variance of the local estimators. This trade-off also influences the choice of hyperparameters, such as the regularization hyperparameter or the truncation levels, which control the smoothness of the estimator.

The performance of statistical procedures in high-dimensional or nonparametric settings often hinges on how well the tuning parameters are chosen. Selecting tuning parameters that allow too much flexibility can lead to high variance due to overfitting, whereas overly restrictive choices may lead to underfitting the true signal. This fundamental tension is classically known as the bias-variance trade-off. In distributed settings, the trade-off becomes even more delicate. When data are split across many machines, each local machine operates on a smaller dataset, and naive tuning can lead to either high bias due to over-smoothing or high variance if the model is made too flexible. \citet{Zhang2015DivideRates} suggested that the local machines should deliberately under-smooth by choosing regularization parameters that allow more model complexity than what would be optimal for their sample size, to control the bias. Although this increases the local variance, averaging across machines reduces the variance at the global level. Most of the existing distributed methods \citep{Zhang2015DivideRates, Shang2017ComputationalSpline, Chang2017DistributedRegression, Mucke2018ParallelizingAlgorithms,  Szabo2019AnMethods} assume knowledge of the underlying smoothness of the function to tune the hyperparameter optimally in order to achieve the bias-variance trade-off. However, in practice, the underlying smoothness is rarely known in advance. This raises a fundamental question: Can we develop a nonparametric distributed method that leverages data-driven tuning to achieve an optimal bias-variance trade-off?  

Data-driven selection of hyperparameters in distributed settings is challenging. Since the optimal choice often depends on the unknown smoothness of the underlying function, a naive approach would require aggregating the entire data at the central machine, a solution that is both communication intensive and practically infeasible. This motivates the need for a distributed procedure that enables data driven tuning of the hyperparameter without full data centralization, while still attaining statistical optimality. A recent work \citep{Szabo2025AdaptationProcesses} addresses this challenge in a spatially distributed regression framework, where the data are partitioned across machines based on location, rather than being randomly split across machines. In this setup, local Gaussian Process models are fitted on spatial sub-regions and aggregated via a mixture-of-experts strategy, where the weights depend on the input location. However, the methodology depends on the assumption that the data are inherently spatial and can be meaningfully partitioned along location-based boundaries, which may not hold true in more general distributed environments.

In this work, however, we consider the more general setting where the data of size $n$ are randomly partitioned across $m$ machines, without considering any spatial alignment. Our goal is to obtain a distributed nonparametric method that achieves this optimal bias-variance trade-off in terms of the mean squared error, without the knowledge of the underlying smoothness $s>0$ of the function. It is intuitive that if we do not have any restrictions on the communication between the local machines and the central machine, then we can essentially end up sending every observation to the central machine and carry out a non-distributed rate-optimal, adaptive estimation method in the central machine. On the other hand, if we have too little communication between the machines, then the best possible rate corresponds to the minimax rate with sample size $n/m$, as discussed in \citet{Szabo2020AdaptiveConstraints}, and hence it would be better to use a single machine. The situation becomes particularly interesting in the intermediate communication regime \citep{Zhu2018DistributedConstraints, Szabo2020AdaptiveConstraints}. A similar three-phase trade-off between communication budget and estimation accuracy was established in the parametric setting by \citet{Cai2024DistributedAlgorithms}, who showed that the minimax risk exhibits a localization phase, a refinement phase and an optimal-rate phase, with sharp thresholds separating these regimes. In the nonparametric setting, \citet{Szabo2020AdaptiveConstraints} showed that a rate-optimal distributed estimation strategy is indeed possible under one round of communication setup, provided at least (up to a logarithmic factor) $n^{1/(1+2s)}$ bits are allowed to be communicated from the local to the central machine, where s denotes the regularity of the functional parameter of interest. Transmitting less bits results in (polynomially) sub-optimal convergence rate. When $s$ is unknown, adaptation is possible over an interval $(s_{\min}, s_{\max})$ if the communication budget, $B$, is large enough, $B\geq n^{1/(1+2s_{\min})}$.  An analogous effect was observed by \citet{Cai2022DistributedAdaptation} in the parametric Gaussian mean model with unknown variance, where estimating the nuisance parameter similarly increases the communication required to attain the optimal rate. \citet{Cai2022DistributedAdaptationb} investigated adaptation in the Gaussian sequence model, deriving the sharp optimal rate and the exact additional communication cost required to achieve it when $s$ is unknown, and also proposed a procedure that attains this bound.  \citet{Szabo2022DistributedCommunication} also considered a one-round communication setting and showed that if the number of machines grows as a polynomial (of order smaller than 1/2, i.e, $\sqrt{n}$) in the total sample size, then there exists a distributed estimator that is adaptive over a range of smoothness levels determined by the order of the polynomial, while communicating the optimal number of bits. The drawback of this method is that if we want to fully utilize the distributed setting, then the range of smoothness levels over which the simultaneous adaptation to the optimal number of bits and the optimal rate is possible, becomes very narrow. Conversely, if we want to apply the adaptation over a wide range of smoothness levels, then a large number of machines cannot be used, thereby impeding the efficiency of the distributed method.

In this paper, we consider the problem of distributed signal estimation under the distributed signal-in-white-noise model introduced by \citet{Szabo2019AnMethods}. The ideas developed in this idealized yet analytically tractable framework lay a rigorous theoretical foundation for adaptive distributed inference in more practical and widely used models, such as the Gaussian Process (GP) regression. Distributed GP methods have been proposed in applied domains due to their flexibility and capacity for uncertainty quantification \citep{Guhaniyogi2022DistributedPrior, Guhaniyogi2023DistributedData}, although they do not adapt to unknown smoothness levels. In this paper, we propose a two-round communication procedure that achieves the minimax optimal estimation without prior knowledge of the underlying smoothness (Besov regularity) $s>0$ of the true signal, assuming that the true signal lies in the Besov space $B^s_{2,\infty}$. In the first round of communications, some summary statistics are sent from the local to the central machine in order to estimate a hyperparameter that depends on the underlying smoothness of the function. In the second round, using the knowledge of this estimated hyperparameter, another set of summary statistics are sent from the local machines, and the central machine aggregates these to form the final estimator. We show that this procedure, under mild conditions that the smoothness $s$ lies in an interval $(s_{\min},s_{\max})$, attains the optimal convergence rate of $n^{-s/(1+2s)}$ and with probability tending to one as $n\to\infty$, transmits the optimal number of bits (up to a logarithmic factor) per machine. The major contributions of this paper can be summarized as:
\begin{itemize}
    \item \textbf{Two round, task-separated communication:} Unlike one-round methods that attempt to compress all information at once, our approach deliberately separates communication into two stages-first, estimating the key hyperparameter, and then performing aggregation. This sequencing ensures that the final summaries are better informed and more accurate. While a form of two step communication is implicitly present in some divide and conquer methods \citep{Ng2014HierarchicalRegression,Jordan2019Communication-EfficientInference} as they often require an additional level of communication for tuning the hyperparameters, our procedure implements this in a systematic way by transmitting suitably chosen summary statistics, thereby achieving both optimal communication and optimal recovery. 
    \item \textbf{Broad adaptivity with practical implementation:} Our method transmits the optimal number of bits and adapts to a wider range of smoothness levels using estimators that are not only easy to compute but also have clear interpretability, making the procedure theoretically appealing and practically deployable.
    \item \textbf{Paradigm-agnostic inference:} We establish that the proposed method achieves optimal statistical guarantees under both frequentist and Bayesian frameworks. Its performance does not rely on a specific inferential paradigm, which makes it broadly applicable and robust across theoretical settings.
\end{itemize}

The paper is organized as follows. Section~\ref{sec:model-section} introduces the distributed signal-in-white noise model. After reviewing a non-adaptive distributed Bayesian method in Section~\ref{sec:nonAdaptive} as a benchmark, Section~\ref{sec:adaptive} presents our proposed adaptive procedure and establishes its theoretical properties under both frequentist and Bayesian perspectives. Section~\ref{sec:Simulation} evaluates its performance through simulation studies, and Section~\ref{sec:Conclusion} concludes with a discussion and directions for future research. The proofs are provided in Appendix.

\section{Distributed signal-in-white noise model}\label{sec:model-section}
We consider the problem of estimating an unknown function $ f_0 \in L^2[0,1] $ from noisy observations in the Gaussian white noise model. In the centralized setting, one observes a continuous-time process $X = (X_t : t \in [0,1])$ satisfying:
\[
X_t = \int_0^t f_0(s)\,ds + \frac{1}{\sqrt{n}} W_t,
\]
where $W$ is a standard Brownian motion, and $n$ represents signal-to-noise ratio. The parameter $n$ determines the difficulty of the estimation problem and is commonly interpreted as a proxy for sample size in this framework \citep{Tsybakov2009IntroductionEstimation, Gine2015MathematicalModels}.

Projecting the signal $f_0$ onto an orthonormal wavelet basis $\{ \psi_{ij} \}_{i \geq 0, j = 0, \dots, 2^i-1}$ of $L^2[0,1]$ (see Section~\ref{sec:wavelet} for details) as
\[
f_0(t) = \sum_{i=0}^\infty \sum_{j=0}^{2^i-1} \theta_{0,ij} \psi_{ij}(t),
\]
the problem of estimating the true signal $f_0$ from the data $X$ boils downs to the problem of estimating the sequence of wavelet coefficients $\theta_0=(\theta_{0,ij})_{i\geq0,\,j=0,\ldots,2^i-1}$ from the noisy observations $\{y_{ij}\}_{i\geq0,j=0,\ldots,2^i-1}$ satisfying:
\[
y_{ij} = \theta_{0,ij} + \frac{1}{\sqrt{n}} Z_{ij}, \quad Z_{ij} \overset{\text{iid}}{\sim} \mathcal{N}(0,1).
\]

We consider the distributed version of this signal-in-white noise model as introduced in \citet{Szabo2019AnMethods}. We assume that each of the $m$ machines observe a noisy version of the true coefficient $\theta_{0,ij}$ independently and in the $k^{th}$ machine, we observe $Y^{(k)}=(y_{ij}^{(k)})_{i\geq 0\,j=0,\ldots,2^i-1}, k=1,\ldots,m$ satisfying
\begin{align} \label{model}
y_{ij}^{(k)}=\theta_{0,ij}+\sqrt{\frac{m}{n}}Z_{ij}^{(k)};i\geq 0;\, j=0,\ldots,2^i-1,
\end{align}
where $Z_{ij}^{(k)}\stackrel{iid}{\sim}N(0,1)$. We denote the total collection of all observations by $Y^{(1:m)}=\left(Y^{(1)},\ldots,Y^{(m)}\right)$. Typically, it is assumed that $\theta_0$ belongs to some regularity class, and in our analysis we consider the Besov space:
\begin{align}
    \label{besov}
    B_{2,\infty}^{s}(L)=\bigg\{\theta\in\ell_2(L):\, \sup_{i\geq 0} 2^{2is}\sum_{j=0}^{2^i-1}\theta_{0,ij}^2\leq L \bigg\},
\end{align}
which is closely related to the Sobolev regularity class $S^{s}(L)$. In most practical scenarios, this smoothness parameter $s$ is not known and our goal is to construct a distributive, adaptive Bayesian procedure which estimates $\theta_0$ with the minimax rate without using any information about the true regularity parameter~$s$. However, as shown by \citet{Szabo2019AnMethods}, standard frequentist and Bayesian methods such as the distributed version of the empirical Bayes approach based on maximum marginal likelihood that perform well in the non-distributed setting can fail to achieve adaptation in the distributed setting. This motivates the development of an alternative empirical Bayes strategy better suited to the distributed framework.

\section{Non-adaptive distributed method}\label{sec:nonAdaptive}
We begin with a non-adaptive distributed Bayesian procedure that assumes prior knowledge of the underlying smoothness of the true signal. We propose the following prior on the wavelet coefficients that retains all the coefficients up to a chosen resolution level $\ell$: 
\begin{align*}
&\theta_{ij}^{(k)}\sim \mathds{1}_{[0:\ell]}(i)\cdot\mathcal{N}(0,1) + (1-\mathds{1}_{[0:\ell]}(i))\cdot\delta_0;\quad j=0,\ldots,2^i-1,\; k=1\ldots,m,
\end{align*}
where $\mathds{1}_{[0:\ell]}(i)$ is the indicator of the event $i\in\{0,\ldots,\ell\}$ and $\ell$ is the hyper-parameter of the prior to be chosen by the user.  

It is well known \citep{Szabo2019AnMethods} that in the distributed Bayesian setting, naively averaging subset posteriors without any correction can result in poor approximations of the full (non-distributed) posterior, leading to suboptimal contraction rates and unreliable uncertainty quantification. Two broad strategies for combining local posteriors have been explored in the literature: the product-of-experts (PoE) approach \citep{Hinton1999ProductsExperts,Scott2016BayesAlgorithm, Srivastava2015WASP:Posteriors}, where the global posterior is formed by multiplying local posteriors, and the mixture-of-experts (MoE) approach \citep{Kim2005AnalyzingProcesses,Vasudevan2009GaussianTerrain,Szabo2025AdaptationProcesses}, where the input space is partitioned across machines and local posteriors are combined piecewise over their respective regions.  The former typically requires correction to prevent over-concentration; common choices include raising the likelihood to the $m^{th}$ power \citep{Srivastava2015WASP:Posteriors, Guhaniyogi2023DistributedData} or by rescaling the Gaussian prior by $m$, as done in consensus Monte Carlo \citep{Scott2016BayesAlgorithm}. In contrast, recent work by \citet{Szabo2025AdaptationProcesses} shows that under spatial data partitioning, mixture-of-experts formulations can achieve adaptation without such corrections.  In this paper, we adopt the product-of-experts framework and consider rescaling the local priors accordingly. The resulting local posteriors are:
\[
\theta_{ij}^{(k)}\big|Y^{(k)}\sim \mathds{1}_{[0:\ell]}(i)\cdot\mathcal{N}\left(\frac{n}{n+1}y_{ij}^{(k)},\frac{m}{n+1}\right)+\big(1-\mathds{1}_{[0:\ell]}(i)\big)\cdot\delta_0;\quad j=0,\ldots,2^i-1,\; k=1,\ldots,m.
\]

Following the consensus Monte Carlo approach, where a draw from the global posterior is taken to be the average of single draws taken from each local posterior, the aggregated global posterior takes the form:
\[
\widetilde{\theta}_{ij}\big|Y^{(1:m)}\sim \mathds{1}_{[0:\ell]}(i)\cdot\mathcal{N}\bigg(\frac{n}{m(n+1)}\sum_{k=1}^my_{ij}^{(k)},\frac{1}{n+1}\bigg)+\big(1-\mathds{1}_{[0:\ell]}(i)\big)\cdot\delta_0; j=0,\ldots,2^i-1,\; k=1,\ldots,m.
\]

Balancing the squared bias and the variance shows that, for $\theta_0\in B_{2,\infty}^{s}(L)$, choosing $\ell$ to be of the order of $\log_2(n)/(1+2s)$ yields posterior contraction around the true $\theta_0$ with the optimal rate $n^{-s/(1+2s)}$ \citep{Gine2015MathematicalModels}. While this construction provides a valid posterior, its performance depends on the correct knowledge of the underlying smoothness of the function $s>0$. However, in general, the regularity hyperparameter $s$ is unknown, which might lead to a miscalibration of the threshold $\ell$, and hence, one has to consider a data-driven choice for this hyperparameter $\ell$.

\section{Adaptive distributed methods}\label{sec:adaptive}
\subsection{Frequentist adaptation}\label{sec:freq}
In this section, we first develop a fully frequentist procedure for adaptive estimation under the distributed signal-in-white-noise model. The frequentist construction serves as a useful stepping stone as it provides theoretical clarity and helps to better understand the difficulties of estimation when the smoothness of the underlying function is unknown. We begin by analyzing a simple thresholding estimator with a fixed thresholding level $\ell$. Specifically, we consider local estimators of the form
\begin{align}\label{Thresholding}
&\hat\theta_{ij}^{(k)}=y_{ij}^{(k)}\mathds{1}_{[0:\ell]}(i),\quad j=0,\ldots,2^i-1, k=1,\ldots,m.
\end{align}

This projection-based approach retains all coefficients up to the level $\ell$, setting all higher coefficients to zero. The main idea behind thresholding at level $\ell$ is to control the trade-off between bias and variance in estimating the underlying function. Lower resolution coefficients capture the main structure of the signal, whereas higher resolution coefficients are often dominated by noise. By truncating at an appropriately chosen level, one can reduce the estimation variance without incurring significant bias. In the non-adaptive setting in Section~\ref{sec:nonAdaptive}, this level $\ell$ is fixed based on known smoothness. However, the goal here is to select $\ell$ in a data-driven manner to adapt to unknown smoothness while preserving the optimal minimax convergence rate.

In \citet{Szabo2019AnMethods}, it was shown that the use of standard maximum marginal likelihood methods for estimating the hyperparameters may lead to suboptimal results. To address this, we introduce a novel data-driven approach to estimate the hyperparameter $\ell$. We consider an interactive two-round communication strategy: in the first step, local machines send some summary statistics to the central machine, where this information is aggregated to estimate the hyperparameter. The resulting quantity is then transmitted back to the local machines. In the second step, the local machines, using the received information, transmit another set of summary statistics to the central machine, which are then aggregated to produce the final distributed thresholding estimator.

A natural question that arises is what kind of summary statistics the local machines should send to the central machine in the first round. As shown in \citet{Szabo2019AnMethods}, marginal likelihood-based summary statistics are not appropriate, as they can lead to suboptimal contraction rates. Sending the complete set of observations or a high-dimensional linear transformation of it,  would impose a huge communication burden, practically reducing the setup to a non-distributed setting. We propose the average of the squared observations at each resolution level as the candidate statistic. Each local machine computes and transmits:
$T_i^{(k)}=\frac{1}{2^i}\sum_{j=0}^{2^i-1}(y_{ij}^{(k)})^2;\; i=0,1,\ldots I,\, j=0,\ldots,2^i-1,\, k=1,\ldots,m$, for some large $I$. These local statistics are then averaged at the central machine to obtain:
\begin{align*}
T_i&=\frac{1}{m}\sum_{k=1}^{m}T_i^{(k)}=\frac{1}{2^i}\sum_{j=0}^{2^i-1}\bigg( \frac{1}{m}\sum_{k=1}^{m}(y_{ij}^{(k)})^2\bigg).
\end{align*}

In view of model \eqref{model}, we can further decompose it as:
\begin{align*}
T_i&=\frac{1}{2^i}\sum_{j=0}^{2^i-1}\theta_{0,ij}^2+\frac{1}{2^i\sqrt{n}}\sum_{j=0}^{2^i-1}\theta_{0,ij}\frac{1}{\sqrt{m}}\sum_{k=1}^{m}Z_{ij}^{(k)}+ \frac{1}{2^i}\sum_{j=0}^{2^i-1} \frac{1}{n}\sum_{k=1}^m(Z_{ij}^{(k)})^2.
\end{align*}

Note that the first term is the average energy at the $i$th resolution level, the second term is a zero mean Gaussian noise term with variance $ n^{-1}2^{-2i}\sum_{j=0}^{2^i-1}\theta_{0,ij}^2$, while the third term is a scaled chi-squared distribution, $(n2^{i})^{-1}\chi^2_{m2^i}$, which has mean $m/n$ and variance $2mn^{-2}2^{-i}$. We define the adjusted summary statistic as:
$$\widetilde{T}_i=2^iT_i-2^i m/n,$$
which has expected value $\sum_{j=0}^{2^i-1}\theta_{0,ij}^2$, and standard deviation of the order $n^{-1/2}(\sum_{j=0}^{2^i-1}\theta_{0,ij}^2)^{1/2}+ 2^{i/2}m^{1/2}/n$. This adjustment effectively isolates the signal energy (i.e. $\ell_2$-norm) at each resolution level while accounting for the noise structure, making $\widetilde{T}_i$ suitable for our adaptive thresholding procedures.

To adaptively select the truncation level $\ell$ without prior knowledge of the underlying smoothness, we develop a procedure inspired by the classical method in \citet{Lepskii1991OnNoise}. Originally proposed in the context of non-distributed Gaussian white noise models, Lepskii's method selects the largest resolution level beyond which the finer scale estimators are statistically indistinguishable from the coarser scale estimators, preventing overfitting. We follow a similar philosophy in our hyperparameter estimation procedure. Since in our distributed setting, estimation is based on aggregated energy statistics rather than direct access to the full signal, we consider a modification of the Lepskii's method \citep{Gine2015MathematicalModels} to select the smallest resolution level $\ell$ beyond which the cumulative signal energy remains below a noise-dominated threshold at all finer levels.

If the unknown smoothness level $s$ is assumed to belong to the interval $(s_{\min},s_{\max})$, for some $0<s_{\min}<s_{\max}$ and that the number of machines is bounded from above by $m\leq n^{1/(1+2s_{\max})}$, then the estimated truncation level $\hat{\ell}$ is defined as:
\begin{align}
\hat{\ell}=\min\bigg\{\ell\in \mathcal{L}:\, \sum_{i=\ell}^{l}\widetilde{T}_i\leq \tau 2^l/n,\,\forall l\geq \ell, l\in\mathcal{L}  \bigg\}\wedge \frac{\log_2(n)}{1+2s_{\min}},\label{def: hyper:estimator}
\end{align}
where $\mathcal{L}=\left[\frac{\log_2 (n)}{1+2s_{\max}}, \frac{\log_2 (n)}{1+2s_{\min}}\right]$, and $\tau>2$. The notational convention $\emptyset\wedge c=c$ ensures that the estimator is well defined. We broadcast this estimated thresholding level to all the local machines. 

In the second round of communication, each machine transmits the wavelet coefficients up to the resolution level $\hat{\ell}$. The central machine then aggregates them by averaging over the machines to get the final distributed estimator of the coefficients as:
\[\hat{\theta}_{ij}=\frac{1}{m}\sum_{k=1}^m y_{ij}^{(k)}\mathds{1}_{[0:\hat{\ell}]}(i);\;j=0,\ldots,2^i-1.\]

Since, we denote the functional form of the parameter $\theta_0$ by $f_0$, the final distributed estimator of the function is given by:
\begin{align}
\hat{f}_n=\sum_{i=0}^{\hat{\ell}}\sum_{j=0}^{2^i-1}  \bigg(\frac{1}{m}\sum_{k=1}^m y_{ij}^{(k)} \bigg)\psi_{ij}.\label{def: distr:freq}
\end{align}

The two rounds of communication, as discussed above, are illustrated in Figure~\ref{fig:TwoRounds}. We now investigate the $L_2$-risk of this estimator under the assumption that the true coefficients $\theta_0\in B_{2,\infty}^s(L)$, for some $s\in(s_{\min},s_{\max})$.

\begin{theorem}\label{thm: freq:adapt}
The distributed estimator \eqref{def: distr:freq} adapts to the minimax rate, i.e. for every $s\in(s_{\min},s_{\max})$
\begin{align*}
\sup_{f_0\in B_{2,\infty}^{s}(L)}E_{f_0}\big\|\hat{f}_n-f_0 \big\|_2^2\leq Cn^{-2s/(1+2s)},
\end{align*}
for some constant $C>0$.
\end{theorem}

\begin{figure}[H]
    \centering
    \begin{subfigure}[h]{0.3\textwidth}
        \centering
        \includegraphics[width=\textwidth]{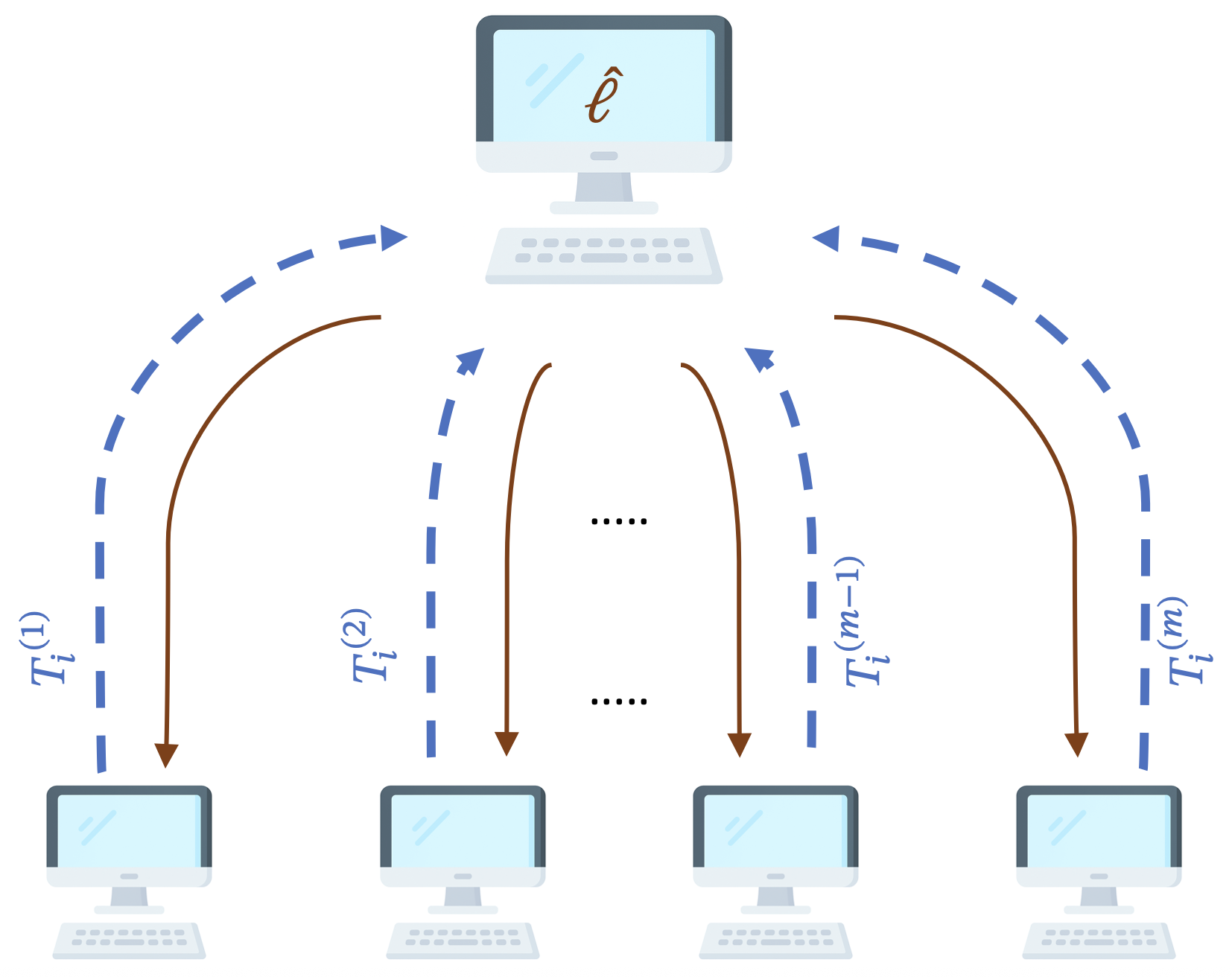}
        \subcaption{Round 1}
        \label{fig:Round1}
    \end{subfigure}\hspace{1.8cm}
    \begin{subfigure}[h]{0.3\textwidth}
        \centering
        \includegraphics[width=\textwidth]{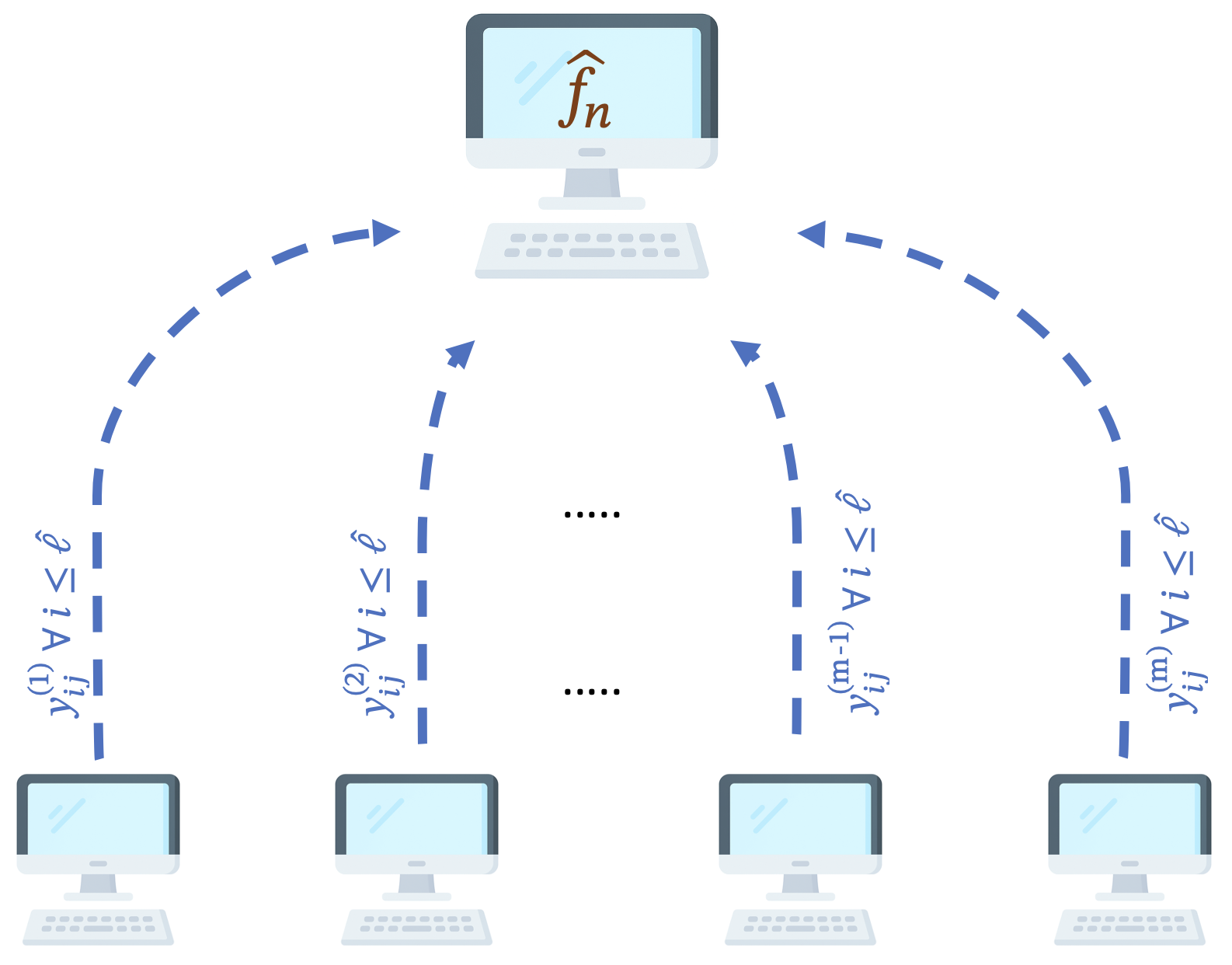}
        \subcaption{Round 2}
        \label{fig:Round2}
    \end{subfigure}
    \caption{Two rounds of communication}
    \label{fig:TwoRounds}
\end{figure}

The proposed adaptive distributed strategy transmits, up to a logarithmic factor, the optimal amount of information \citep{Szabo2022DistributedCommunication}, in the sense that for any $s> s_{\min}$ the total communication cost per local machine is, with probability tending to one as $n\to\infty$, of the order $n^{1/(1+2s)}$ times at most a logarithmic factor. In our setting, communication takes place in two rounds: the first transmits a summary statistic used to estimate the truncation level, and the second transmits empirical wavelet coefficients truncated at the estimated level. Each bounded real-valued quantity can be approximated with precision $n^{-C}$ using $O(C\log n)$ bits \citep{Szabo2020AdaptiveConstraints}, and we use this to quantify the total communication cost. The following result formalizes this communication guarantee.
\begin{theorem}\label{bits}
The interactive distributed method proposed in $ \eqref{def: distr:freq}$ satisfies that the number of bits $\hat{B}_i$ transmitted between the $i$th server and the central server is optimal, i.e. for every $s\in (s_{\min},s_{\max})$ we have that
\begin{align*}
\inf_{f_0\in B_{2,\infty}^s(L)}P_{f_0}\Big(\hat{B}_i\leq C \log (n) n^{1/(1+2s)}\Big)\stackrel{n\rightarrow \infty}{\rightarrow} 1.
\end{align*}
\end{theorem}

\subsection{Bayesian adaptation}\label{sec:Bayes}
We now return to the Bayesian framework and develop an adaptive procedure that builds on the prior structure introduced in Section~\ref{sec:nonAdaptive}. This prior places $N(0,1)$ distributions on the wavelet coefficients up to a truncation level $\ell$, and sets the remaining coefficients to zero. We adopt an empirical Bayes strategy here and estimate the truncation level $\ell$ using the data-driven rule defined in equation~\eqref{def: hyper:estimator}, which was motivated from the frequentist perspective in the previous section.

The construction of the adaptive aggregated posterior $\widetilde{\Pi}$ proceeds via a two round communication procedure, as before. The first round of communication is identical to the one described in Section~\ref{sec:freq} and we estimate the hyperparameter $\ell$, resulting in an estimator $\hat{\ell}$, which is then broadcast back to the local machines. The resulting local posteriors are:
\[
\theta_{ij}^{(k)}\big|Y^{(k)}\sim \mathds{1}_{[0:\hat{\ell}]}(i)\cdot\mathcal{N}\bigg(\frac{n}{n+1}y_{ij}^{(k)},\frac{m}{n+1}\bigg)+\big(1-\mathds{1}_{[0:\hat{\ell}]}(i)\big)\cdot\delta_0;\quad j=0,\ldots,2^i-1,\; k=1,\ldots,m.
\]
Since the posterior distribution here is fully characterized by its mean, each machine transmits only the posterior means of wavelet coefficients up to level $\hat{\ell}$ to the central machine in the second round of communication. These are then averaged at the central machine to form the final data-driven aggregated posterior $\widetilde{\Pi}$:
\begin{align}
\widetilde{\theta}_{ij}\big|Y^{(1:m)}\sim \mathds{1}_{[0:\hat{\ell}]}(i)\cdot\mathcal{N}\bigg(\frac{n}{m(n+1)}\sum_{k=1}^my_{ij}^{(k)},\frac{1}{n+1}\bigg)+\big(1-\mathds{1}_{[0:\hat{\ell}]}(i)\big)\cdot\delta_0;\quad j=0,\ldots,2^i-1.\label{def: distr:adapt:post}
\end{align}

In practice, we truncate the infinite-dimensional index $i$ at a finite level $I$, which represents the maximum resolution level considered in the implementation. A complete description of the empirical Bayes procedure based on two rounds of communication is provided in Algorithm~\ref{algo}. The following result shows that the adaptive aggregated posterior $\widetilde{\Pi}$ contracts around the true functional parameter of interest $f_0$ at the optimal rate.

\begin{algorithm}[h]
\caption{Two rounds of communication}
\label{algo}
    \KwIn{ $y_{ij}^{(k)}$ as data in each local machine $k=1,\ldots,m$,\;$i=0,\ldots,I,\;j=0,\ldots,2^i-1$}
    Initialize $\ell_{\min} = \texttt{ceil}( \log_2(n) / (1+2 s_{\max}) )$\\
    Initialize $\ell_{\max} = \texttt{floor}( \log_2(n) / (1+2 s_{\min}) )$\\
    Initialize $\mathcal{L} = \{\ell_{\min},\ell_{\min}{+}1,\ldots,\ell_{\max}\}\cap \{0,1,\ldots,I\}$\\
    
    \Comment{Round 1: hyperparameter estimation}\\
    \For{$k=1$ \KwTo $m$}{ 
    \For{$i\in\mathcal{L}$}{
    Compute $T_i^{(k)}=\frac{1}{2^i}\sum_{j=0}^{2^i-1}(y_{ij}^{(k)})^2$\\
    Transmit $T_i^{(k)}$ to the central machine
    }
    }
    \Comment{At the central machine}\\
    \For{$i\in\mathcal{L}$}{
    Compute $T_i=\frac{1}{m}\sum_{k=1}^mT_i^{(k)}$\\
    Compute $\widetilde{T}_i=2^iT_i-2^i\frac{m}{n}$
    }
    Apply Lepskii's method to get the estimate $\hat{\ell}$\\
    \Comment{Round 2: Posterior mean transmission}\\
    \For{$k=1$ \KwTo $m$}{
    Transmit $\frac{n}{n+1}y_{ij}^{(k)}\;\forall i=0,\ldots,\hat{\ell},j=0,\ldots,2^i-1$ to the central machine
    }
    \Comment{At the central machine}\\
    Compute $\widetilde{y}_{ij}=\frac{1}{m}\sum_{k=1}^m\frac{n}{n+1}y_{ij}^{(k)}\;\forall i=0,\ldots,\hat{\ell},j=0,\ldots,2^i-1$\\
    \KwOut{Aggregated posterior $\widetilde{\Pi}$ characterized by means $\widetilde{y}_{ij}$ and variance $1/(n+1)$.} 
\end{algorithm}

\begin{theorem}\label{thm:Bayes:adapt}
The distributed empirical Bayes posterior distribution \eqref{def: distr:adapt:post} adapts to the unknown regularity parameter $s$ and achieves minimax posterior contraction rate, i.e. for every $s\in(s_{\min},s_{\max})$
\begin{align*}
\sup_{f_0\in B_{2,\infty}^{s}(L)}E_{f_0}  \widetilde\Pi \Big (f:\, \big\|f-f_0\big\|_2^2\geq M_n  n^{-2s/(1+2s)}\big|Y^{(1:m)}\Big)\to 0,
\end{align*}
as $n\to\infty$ for arbitrary sequence $M_n$ tending to infinity.
\end{theorem}

While the theoretical analysis was guided by the frequentist analogue, the overall methodology is agnostic to the Bayesian or frequentist paradigm. The two-round strategy is modular and can be applied in either framework. In particular, if uncertainty quantification is of interest, the Bayesian formulation naturally yields credible intervals or posterior summaries through aggregation. However, valid uncertainty quantification is generally difficult to achieve in adaptive nonparametric Bayesian procedures, as credible sets may fail to achieve correct frequentist coverage without additional conditions such as the polished tail assumption \citep{Szabo2015FrequentistSets, Rousseau2020ASYMPTOTICPRIORS}.

\section{Simulation study}\label{sec:Simulation}
We illustrate the performance of our proposed adaptive distributed empirical Bayes estimator, referred to as aDC henceforth, using synthetic data. The underlying true sequence $ \theta_0 = (\theta_{0,ij}) $ is generated as follows:
\[\theta_{0,ij} = 2^{-is}\frac{\sin(j+1)}{j+1};\quad i=0,1,\cdots,15,j=0,\ldots,2^i-1.\]
The resolution level $i$ is truncated at $I=15$ for computational feasibility. This construction satisfies the condition of the Besov space $B^s_{2,\infty}(L)$, as defined in Section~\ref{besov}. The data are generated using the white noise model described in ~\eqref{model}, with the smoothness parameter set to $s=0.5$. The function is reconstructed from the estimated coefficients using the Cohen, Daubechies and Vial construction of a compactly supported, orthonormal, $N$-regular wavelet basis $\{ \psi_{ij} \}_{i \geq 0, j = 0,\dots, 2^i-1}$ of $L^2[0,1]$ with $N=8$ vanishing moments (refer to Section~\ref{sec:wavelet} for details) as 
\[\tilde{f}(t) = \sum_{i=0}^I \sum_{j=0}^{2^i-1} \tilde{\theta}_{ij} \psi_{ij}(t),\]
where $\tilde{\theta}_{ij}$ is the aggregated posterior mean of $\theta_{ij}$. Its performance is evaluated using mean integrated squared error (MISE), calculated as $\text{MISE}(\tilde{f})=(1/H)\sum_{h=1}^H\big\{\tilde{f}(x_h)-f_0(x_h)\big\}^2$, where $\{x_h\}_{h=1}^H$ is a uniform grid of $H=10,000$ points in $[0,1]$. Boxplots summarize the MISE across 100 replications for each simulation setting.

We first investigate how the performance of aDC varies with the total sample size and the number of machines. We vary the total sample size $n\in\{10000,  25000, 50000,  75000, 100000, 125000\}$ while keeping the number of machines fixed at $m=10$. The corresponding results are in Figure~\ref{fig:SampleSize}. The estimation accuracy increases as the sample size increases. This is expected as increasing the total sample size leads to more data being available per machine, leading to better estimation.

\begin{figure}[h]
    \centering
    \begin{subfigure}[t]{0.43\textwidth}
        \centering
        \includegraphics[width=.95\textwidth]{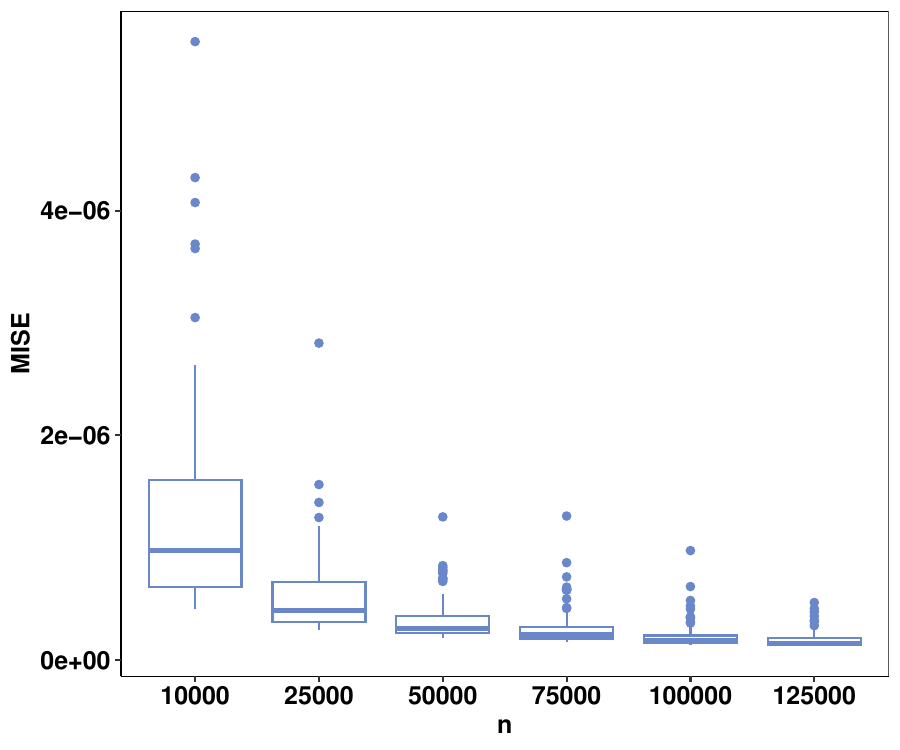}
        \subcaption{Varying total sample size}
        \label{fig:SampleSize}
    \end{subfigure}\hfill
    \begin{subfigure}[t]{0.5\textwidth}
        \centering
        \includegraphics[width=.95\textwidth]{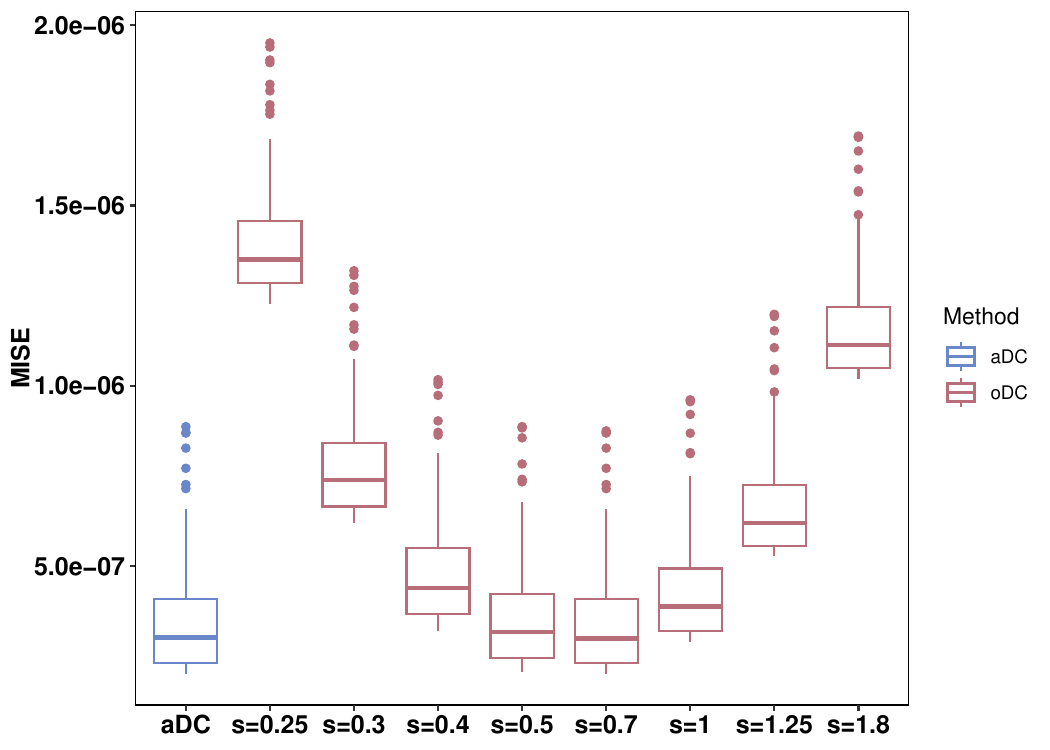}
        \subcaption{Misspecified true smoothness}
        \label{fig:misspecify}
    \end{subfigure}
    \caption{Boxplots of MISEs are based on results from 100 replications}
\end{figure}

To highlight the importance of adaptation in distributed estimation, we next examine the impact of smoothness misspecification on non-adaptive estimators, which rely on oracle knowledge of the underlying regularity. We consider a setting with $n=50000$ samples distributed across $m=10$ machines and the true smoothness is set to $s=0.5$. The adaptive distributed Bayesian estimator (aDC) selects the truncation level via an empirical Bayes rule, while the non-adaptive distributed estimator (oDC), described in Section~\ref{sec:nonAdaptive}, assumes the oracle knowledge of the smoothness and uses the theoretically optimal truncation level to achieve optimal rate. Figure \ref{fig:misspecify} shows boxplots of the MISE for various assumed smoothness value, evaluated over 100 replicates. As the assumed smoothness deviates from the true value, the MISE of the non-adaptive estimator increases rapidly, highlighting its lack of robustness.

We next compare our method (aDC) with the adaptive procedure introduced in \citet{Szabo2022DistributedCommunication}, henceforth referred to as mcDC. The mcDC method uses a one-round communication strategy for distributed function estimation that relies on testing-based techniques to infer the smoothness of the underlying signal. It is shown that simultaneous adaptation to both the minimax estimation rate and the optimal (up to a logarithmic factor) communication budget of order $(L^2n)^{1/(1+2s)}$ is possible only when the smoothness parameter $s$ lies within a restricted range, specifically $s\in[s_{\min},s_{\max}]$ and $s_{\min}<s_{\max}\leq 1/(4p)-1/2$, assuming $m=n^p$ and $m\geq5\log_2(n)$. If $s_{\max}>s_{\min}>1/(4p)-1/2$, there is no one-round distributed method that can, at the same time, transmit the minimal number of bits (up to a small polynomial factor) and achieve the minimax rate (up to a small polynomial factor). We consider both regimes in our simulations to assess the performance and adaptability of our proposed two-round method. 
\begin{figure}[h]
    \centering
    \begin{subfigure}[h]{0.44\textwidth}
        \includegraphics[width=\textwidth]{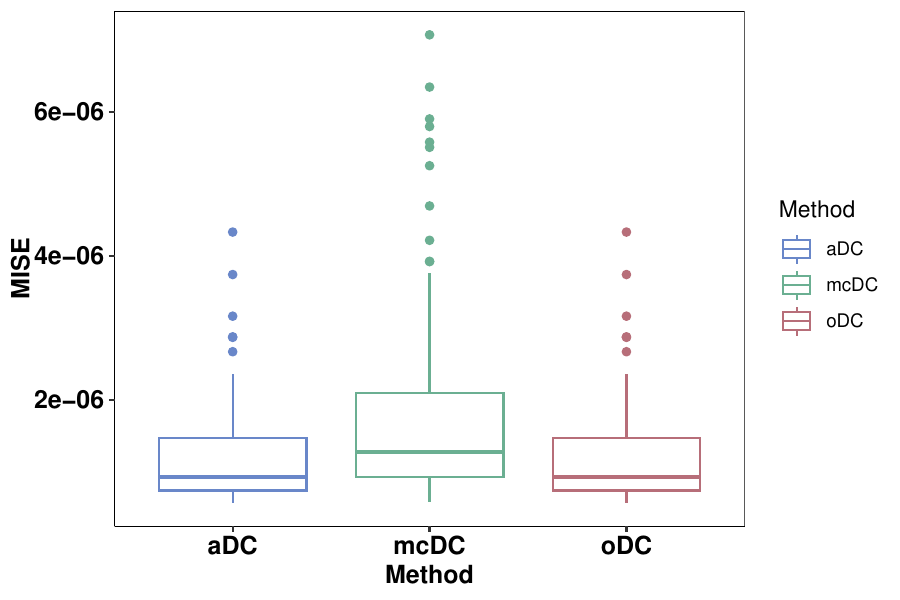}
        \subcaption{$s_{\min}<s_{\max}\leq 1/(4p)-1/2$}
        \label{fig:mcDC1}
    \end{subfigure}\hspace{1.4cm}
    \begin{subfigure}[h]{0.44\textwidth}
        \includegraphics[width=\textwidth]{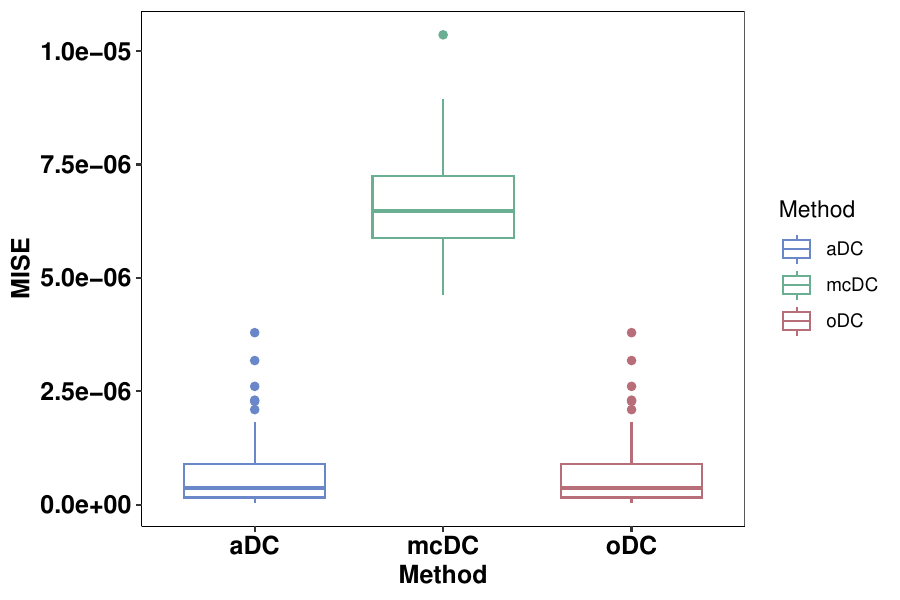}
        \subcaption{$s_{\max}>s_{\min}\geq 1/(4p)-1/2$}
        \label{fig:mcDC2}
    \end{subfigure}
    \caption{Comparison of the performances of aDC, mcDC, and oDC under different scenarios. Boxplots of MISEs are based on results from 100 replications}
\end{figure}

In the first setting, we consider $s=0.45$, $s_{\min}=0.2$ and $s_{\max}=0.5$, with $n=10^4$ total observations and $m=10$ machines. Here, $s_{\min}<s_{\max}\leq 1/(4p)-1/2$, the conditions under which mcDC can theoretically adapt. The corresponding boxplots of MISEs are plotted in Figure~\ref{fig:mcDC1}. In the second setting, we consider $s=1.2$, $s_{\min}=0.51$ and $s_{\max}=1.5$, with $n=10^4$ total observations and $m=10$ machines. This falls in the regime $s_{\max}>s_{\min}\geq 1/(4p)-1/2$, where adaptation is not possible under one round communication. Figure~\ref{fig:mcDC2} shows the corresponding boxplots of the MISEs. Across 100 replications, we compare the estimation error, measured by MISE, of aDC and mcDC and also record the average bits transmitted per machine, counting both rounds for aDC and the coefficient transmissions for mcDC. In the first regime, aDC transmits 3612 bits on average versus 2002 bits for mcDC, while in the second regime, the averages are 520.8 bits and 126 bits respectively. Although mcDC transmits fewer bits per machine, our method (aDC) substantially outperforms mcDC in both regimes, when the tuning parameters in the Lepski-type selection step are calibrated properly. These findings indicate that our two-round communication-based empirical Bayes approach performs reliably across a broad range of smoothness values, including regimes where existing one-round adaptive procedures like mcDC fail to achieve optimal rates. 

Even when communication is not the bottleneck, adaptive methods based on marginal likelihood can still fail under certain signal structures. We illustrate this next by considering a challenging signal in which the energy is concentrated at a single resolution level. As noted in \citet{Szabo2019AnMethods}, such a sparse, localized structure poses difficulties for marginal likelihood-based hyperparameter estimation: although the estimator aggregates information across machines to estimate the truncation level, it fails to capture resolution-specific features critical to recovery. As a result, the estimated signal is over-smoothed and misses key structure. The example illustrated next is inspired by a construction in \citet{Szabo2019AnMethods}, adapted here to evaluate the performance of aDC under extreme sparsity. 

We construct a signal with all coefficients set to zero except at resolution levels $i = 2, 3\ \&\ 4$, with at most 4 active coefficients per level. We consider $\theta_{0,1 1}=0.15,\theta_{0,2 2}=0.135,\theta_{0,3 2}=-0.01,\theta_{0,3 4}=-0.05,\theta_{0,3 7}=0.04$. Data are generated according to the white noise model introduced at the start of this section, with $n = 10^4$ total observations evenly distributed across $m = 10$ machines. We compare the performance of our adaptive distributed estimator (aDC) against the distributed marginal maximum likelihood estimator (MMLE) method described in \citet{Szabo2019AnMethods}, which selects the truncation level by maximizing a marginal likelihood aggregated over machines. For both methods, the signal was reconstructed on a uniform grid of $H=65536$ points over the interval $[0, 1]$. The MISE for aDC was $2.52\times 10^{-8}$, compared to $3.55 \times 10^{-7}$ for the MMLE method. As shown in Figure~\ref{fig:sparse_signal}, the MMLE-based estimator smooths over the localized signal and fails to recover its key features, while our method accurately captures the true structure.

\begin{figure}[h]
    \centering
        \includegraphics[width=0.66\linewidth]{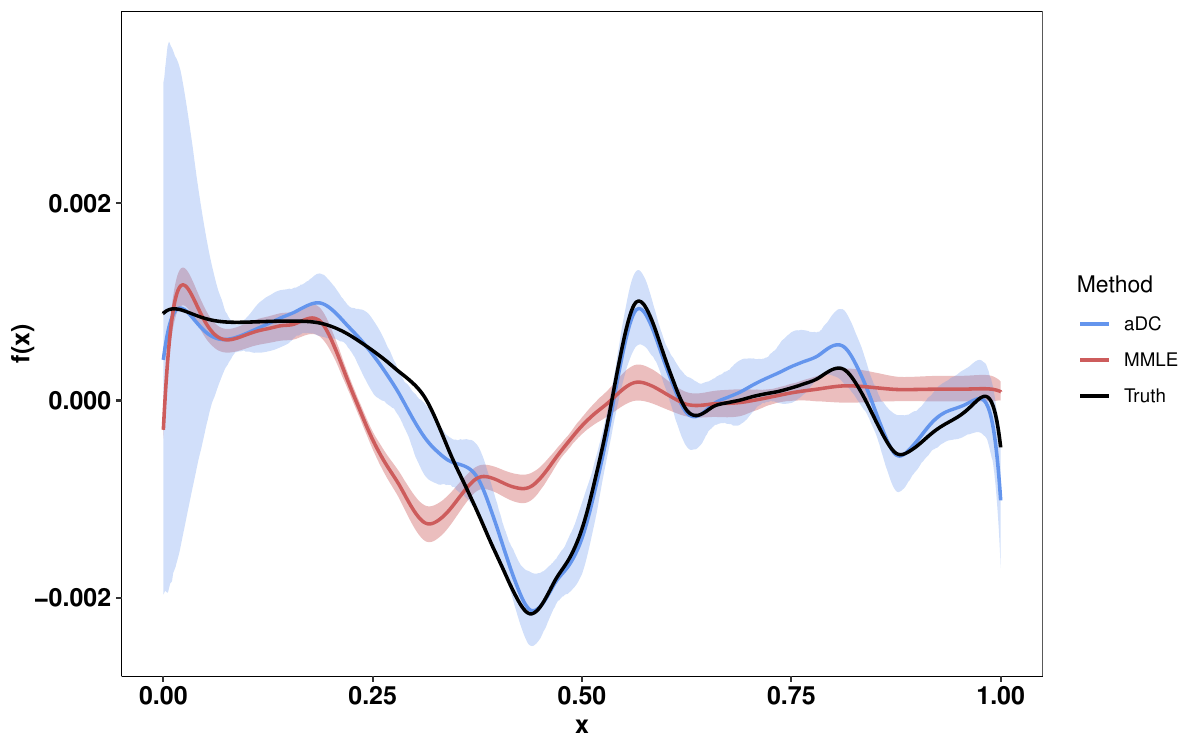}
    \caption{Comparison of aDC and MMLE methods for a sparse signal with energy concentrated at selected resolution levels.}
    \label{fig:sparse_signal}
\end{figure}

\section{Conclusion}\label{sec:Conclusion}
We proposed a novel two-round adaptive distributed estimation procedure for the signal-in-white-noise model. In contrast to one-round methods that attempt to compress all information at once, our approach decouples communication into two stages: an initial round for hyperparameter estimation, followed by a second round that aggregates local summaries using the estimated hyperparameter. The resulting estimator achieves optimal rate of convergence adaptively over a wide range of unknown smoothness levels without specifying tuning parameters. It is computationally efficient, and is agnostic to the inferential paradigm, achieving strong performance under both Bayesian and frequentist formulations. The simulations demonstrate strong empirical performance, particularly in scenarios where existing methods break down.

A natural direction for future work is to extend this two-round adaptive communication framework to more complex models, such as distributed Gaussian process regression. Unlike the signal-in-white-noise setting, GP models involve correlated observations and unknown covariance structures, introducing new challenges for distributed inference. Extending our approach to these settings holds promise for scalable and adaptive learning in spatial and functional data applications.

\bibliographystyle{plainnat}
\bibliography{references}
\newpage
\appendix

\section*{\LARGE Appendix}
\section{Notations}
Let us first summarize the key notations used throughout our paper. In contrast to one-round communication methods, 
\begin{table}[h]
\centering
\caption{Notation reference}
\begin{tabular}{l|l}

\textbf{Symbol} & \textbf{Definition} \\
\hline
$m$ & Number of machines used\\
$n$ & Total sample size\\
$s$ & True underlying smoothness\\
$s_{\min},s_{\max}$ & The minimum and maximum values (respectively) that $s$ is \\
& assumed to take, $s\in(s_{\min},s_{\max})$.\\
$y_{ij}^{(k)}$ & Observation from machine $k$ at resolution level $i$ and location index $j$ \\
$y_{ij}$ & Averaged observation at $(i,j)$, $y_{ij} = \frac{1}{m} \sum_{k=1}^m y_{ij}^{(k)}$ \\
$Y^{(k)}$ & Collection of all the observations stored in machine $k$\\
$Y^{(1:m)}$ & Collection of all the observations across all the machines $k=1,\ldots,m$\\
$\theta_{0,ij}$ & True wavelet coefficient at resolution level $i$ and location index $j$ \\
$Z_{ij}^{(k)}$ & Random noise term at resolution level $i$, location $j$, and machine $k$ \\
$\widetilde{T}_i$ & Signal detection statistic at resolution level $i$,\\
& $\widetilde{T}_i = \sum_{j=0}^{2^i-1} \Big\{ \frac{1}{m} \sum_{k=1}^m \big( Y_{ij}^{(k)} \big)^2 - \frac{2^i m}{n} \Big\}$ \\
$\hat{f}_n(\ell)$ & Empirical estimator, $\hat{f}_n(\ell) = \sum_{i=1}^\ell \sum_{j=0}^{2^i-1} y_{ij} \psi_{ij}$ \\
$c_s$ & $c_s=1/(1-2^{-2s})$\\
$B_n(\ell, f_0)$ & Bias proxy, $c_sB_n(\ell, f_0) = c_s\|f_0\|_{B_{2,\infty}^s}^2 2^{-2\ell s}$ \\
$\|f_0\|_{B_{2,\infty}^s}^2$ & Besov norm, $\|f_0\|_{B_{2,\infty}^s}^2 = \sup_{i \geq 0} 2^{2is} \sum_{j=0}^{2^i-1} \theta_{0,ij}^2$ \\
$\mathcal{L}$& The grid of resolutions used for Lepski, $\mathcal{L}=\Big[\frac{\log_2 (n)}{1+2s_{\max}}, \frac{\log_2 (n)}{1+2s_{\min}}\Big]$.\\
$\hat{\ell}$ & Estimated resolution level,\\
& $\hat{\ell}=\min\big\{\ell\in \mathcal{L}:\, \sum_{i=\ell}^{l}\widetilde{T}_i\leq \tau 2^l/n,\,\forall l\geq \ell, l\in\mathcal{L}  \big\}\wedge \frac{\log_2(n)}{1+2s_{\min}}$\\
$\ell^*$ & Oracle optimal resolution level, $\ell^* = \min \Big\{ \ell\in\mathcal{L}:\, B_n(\ell, f_0) \leq \frac{2^\ell}{n} \Big\}$ \\
\end{tabular}
\end{table}

\section{Proofs}\label{Proof}
\begin{noindentsection}

\subsection{Proof of Theorem \ref{thm: freq:adapt}}\label{sec: prof:thm:freq:adapt}
Let us first note that $2^{2is}\sum_{j=0}^{2^i-1}\theta_{0,ij}^2\leq L$ from the definition of $B_{2,
\infty}^s(L)$, and so we get that $B_n(\ell,f_0)\leq L2^{-2\ell s}$. Note that % Also, $B_n(\ell,f_0)=c'2^{-2\ell s}$ for some suitable constant $c'>0$ depending on $f_0$. 
\begin{align}\label{eq:0}
    \|E_{f_0} \hat{f}_n(\ell)-f_0\|_2^2=\sum_{i>\ell}\sum_{j=0}^{2^i-1}\theta_{0,ij}^2\leq c_s\|f_0\|_{B_{2,\infty}^s}^2\sum_{i>\ell}2^{-2is}\leq B_n(\ell,f_0).
\end{align}

We defined $\ell^*=\min\left\{\ell\in\mathcal{L}:B_n(\ell,f_0)\leq2^\ell/n\right\}$, so
\begin{align*}
&B_n(\ell^*,f_0)=c_s\|f_0\|_{B_{2,\infty}^s}^22^{-2\ell^* s}\leq2^{\ell^*}/n\implies c_1n^{1/(1+2s)}\leq 2^{\ell^*},\text{ and}\\
&c_s\|f_0\|_{B_{2,\infty}^s}^2 2^{-2(\ell^*-1) s}\geq B_n(\ell^*-1,f_0)>2^{\ell^*-1}/n\implies c_2n^{1/(1+2s)}>2^{\ell^*},
\end{align*}
where $c_1=(c_s\|f_0\|_{B_{2,\infty}^s}^2)^{1/(1+2s)}$,  and $c_2=(2^{2s}c_s\|f_0\|_{B_{2,\infty}^s}^2)^{1/(1+2s)}$. We thus have
\[c_1n^{1/(1+2s)}\leq 2^{\ell^*}< c_2n^{1/(1+2s)}.\]

We begin by decomposing the risk into two components. For any (possibly data-driven) choice of $\hat{\ell}$, we have
\begin{align}\label{eq:1}
E_{f_0}\|f_0-\hat{f}_n(\hat{\ell})\|_2^2 = E_{f_0}\big(\|f_0-\hat{f}_n(\hat{\ell})\|_2^2 \cdot\mathds{1}_{\hat{\ell}\leq \ell^*}\big) + E_{f_0}\big(\|f_0-\hat{f}_n(\hat{\ell})\|_2^2 \cdot\mathds{1}_{\hat{\ell}> \ell^*}\big).
\end{align}

We deal with the two terms on the right hand side of \eqref{eq:1} separately. First note that
\begin{align}\label{eq:2}
E_{f_0}\big(\|f_0-\hat{f}_n(\hat{\ell})\|_2^2 \cdot\mathds{1}_{\hat{\ell}\leq \ell^*}\big)
&\leq 2E_{f_0}\big(\|\hat{f}_n(\ell^*)-\hat{f}_n(\hat{\ell})\|_2^2 \cdot\mathds{1}_{\hat{\ell}\leq \ell^*}\big) + 2E_{f_0}\big(\|f_0-\hat{f}_n(\ell^*)\|_2^2 \cdot\mathds{1}_{\hat{\ell}\leq \ell^*}\big).
\end{align}

For the first term in \eqref{eq:2}, note that 
\begin{align}\label{eq:4}
E_{f_0}\big(\|\hat{f}_n(\ell^*)-\hat{f}_n(\hat{\ell})\|_2^2 \cdot\mathds{1}_{\hat{\ell}\leq \ell^*} \big)&=E_{f_0}\bigg(\sum_{i=\hat{\ell}+1}^{\ell^*}\sum_{j=0}^{2^i-1} \Big(\frac{1}{m} \sum_{k=1}^{m}y_{ij}^{(k)}\Big)^2 \cdot\mathds{1}_{\hat{\ell}\leq \ell^*}\bigg)\nonumber\\
&\leq 2 E_{f_0}\Bigg(\sum_{i=\hat{\ell}}^{\ell^*}\sum_{j=0}^{2^i-1} \bigg( \theta_{0,ij}^2+  \Big(\frac{1}{\sqrt{mn}} \sum_{k=1}^{m}Z_{ij}^{(k)}\Big)^2\bigg) \cdot\mathds{1}_{\hat{\ell}\leq \ell^*}\Bigg).
\end{align}

The second term on the right hand side of \eqref{eq:4} is bounded by
\begin{align}\label{eq:6a}
\sum_{i=0}^{\ell^*}\sum_{j=0}^{2^i-1} E_{f_0}  \Big(\frac{1}{\sqrt{mn}} \sum_{k=1}^{m}Z_{ij}^{(k)}\Big)^2=\frac{2^{\ell^*+1}-1}{n}\leq  2^{\ell^*+1}/n\leq c_5 n^{-\frac{2s}{1+2s}},
\end{align}
where $c_5=2c_2$. Hence, it remains to deal with the first term of \eqref{eq:4}. For this, note that when $\hat{\ell}\leq \ell^*$,
\begin{align*}
\sum_{i=\hat{\ell}}^{\ell^*}\sum_{j=0}^{2^i-1}\theta_{0,ij}^2 &=\sum_{i=\hat{\ell}}^{\ell^*}\sum_{j=0}^{2^i-1}\bigg(\frac{1}{m}\sum_{k=1}^m (y_{ij}^{(k)})^2-\frac{m}{n}-\frac{2}{m}\sum_{k=1}^m \sqrt{\frac{m}{n}}Z_{ij}^{(k)}\theta_{0,ij}-\frac{1}{n}\sum_{k=1}^m\big( (Z_{ij}^{(k)})^2-1\big)\bigg)\\
&\leq \sum_{i=\hat{\ell}}^{\ell^*}\widetilde{T}_i+\frac{2}{\sqrt{mn}}\sum_{i=\hat{\ell}}^{\ell^*}\sum_{j=0}^{2^i-1}\bigg|\theta_{0,ij}\sum_{k=1}^m Z_{ij}^{(k)}\bigg|+ \frac{1}{n}\sum_{i=0}^{\ell^*}\bigg|\sum_{j=0}^{2^i-1}\sum_{k=1}^m\left( (Z_{ij}^{(k)})^2-1\right)\bigg|.
\end{align*}

Next, using the fact that $|ab|\leq(a^2+b^2)/2$, we get for $\hat{\ell}\leq \ell^*$,
\begin{align*}
    \frac{2}{\sqrt{mn}}\sum_{i=\hat{\ell}}^{\ell^*}\sum_{j=0}^{2^i-1}\Big|\theta_{0,ij}\sum_{k=1}^m Z_{ij}^{(k)}\Big| &= \sum_{i=\hat{\ell}}^{\ell^*}\sum_{j=0}^{2^i-1}\bigg|\theta_{0,ij}\Big(\frac{2}{\sqrt{mn}}\sum_{k=1}^m Z_{ij}^{(k)}\Big)\bigg|\\
    &\leq \sum_{i=\hat{\ell}}^{\ell^*}\sum_{j=0}^{2^i-1} \frac{1}{2}\bigg( \theta_{0,ij}^2+\frac{4}{mn}\Big(\sum_{k=1}^m Z_{ij}^{(k)}\Big)^2\bigg)\\
    &\leq \frac{1}{2}\sum_{i=\hat{\ell}}^{\ell^*}\sum_{j=0}^{2^i-1} \theta_{0,ij}^2 + \frac{2}{mn}\sum_{i=0}^{\ell^*}\sum_{j=0}^{2^i-1}\Big(\sum_{k=1}^m Z_{ij}^{(k)}\Big)^2.
\end{align*}

By combining the preceding two displays, we get
\begin{align}\label{eq:5}
E_{f_0}\bigg(\sum_{i=\hat{\ell}}^{\ell^*}\sum_{j=0}^{2^i-1}\theta_{0,ij}^2\cdot\mathds{1}_{\hat{\ell}\leq \ell^*}\bigg)
&\leq 2E_{f_0} \Big(\sum_{i=\hat{\ell}}^{\ell^*}\widetilde{T}_i\cdot\mathds{1}_{\hat{\ell}\leq \ell^*} \Big) + \frac{4}{mn} \sum_{i=0}^{\ell^*}\sum_{j=0}^{2^i-1}E_{f_0}\bigg(\sum_{k=1}^{m}Z_{ij}^{(k)}\bigg)^2\nonumber\\
&\qquad\qquad+ \frac{2}{n}\sum_{i=0}^{\ell^*}E_{f_0}\bigg|\sum_{j=0}^{2^i-1}\sum_{k=1}^m\Big( (Z_{ij}^{(k)})^2-1\Big)\bigg|.
\end{align}

We deal with the three terms on the right hand side of \eqref{eq:5} separately. The first term is bounded from above by $\tau 2^{\ell^*}/n$ following from the definition of $\hat{\ell}$. It is easy to see that the second term is again bounded from above by a multiple of $2^{\ell^*}/n$. Finally, for the third term, using Jensen's inequality we can say that the expected value of the absolute value of a centered chi-square distribution is bounded above by its standard deviation. Hence,
\begin{align*}
 \frac{1}{n}\sum_{i=0}^{\ell^*}E_{f_0}\left|\sum_{j=0}^{2^i-1}\sum_{k=1}^m\left( (Z_{ij}^{(k)})^2-1\right)\right|\leq \frac{1}{n}\sum_{i=0}^{\ell^*}(2(m2^i))^{1/2} \leq   \frac{\sqrt{2m}}{n}\sum_{i=0}^{\ell^*} (2^{i})^{1/2}\leq \frac{2}{\sqrt{2}-1}\frac{2^{\ell^*}}{n},
\end{align*}
since $m\leq 2^{\ell_{\min}}\leq 2^{\ell^*}$. Thus, we have
\[E_{f_0}\sum_{i=\hat{\ell}}^{\ell^*}\sum_{j=0}^{2^i-1}\theta_{0,ij}^2 \leq c_62^{\ell^*}/n,\text{ where }c_6=2\tau+4+4/(\sqrt{2}-1).\]

Using \eqref{eq:0} and \eqref{eq:6a}, the second term on the right hand side of \eqref{eq:2} is further bounded by
\begin{align}\label{eq:3}
\|E_{f_0}\hat{f}_n(\ell^*)-f_0 \|_2^2+E_{f_0}\|\hat{f}_n(\ell^*)-E_{f_0}\hat{f}_n(\ell^*) \|_2^2\leq c'L2^{-2\ell^* s}+2^{\ell^*+1}/n\leq c_4 n^{-\frac{2s}{1+2s}},
\end{align}
where $c_4=c_2(Lc'+2)$. We now deal with the second term on the right hand side of \eqref{eq:1}. Using Cauchy-Schwartz inequality, 
\begin{align*}
    E_{f_0}\left(\|f_0-\hat{f}_n(\hat{\ell})\|_2^2 \cdot\mathds{1}_{\hat{\ell}> \ell^*}\right)&=\sum_{\ell=\ell^*+1}^{\ell_{\max}} E_{f_0}\left(\|f_0-\hat{f}_n(\ell)\|_2^2 \cdot\mathds{1}_{\hat{\ell}= \ell}\right)\\
&\leq \sum_{\ell=\ell^*+1}^{\ell_{\max}}\Big(E_{f_0}  \big(\|f_0-\hat{f}_n(\ell)\|_2^2 \big)^2\Big)^{1/2} P^{1/2}_{f_0}(\hat{\ell}=\ell).
\end{align*}

Using the fact that $\sqrt{a^2+b^2}\leq a+b$, when $a,b>0$, we get
\begin{align}\label{eq:7.5}
    \Big(E_{f_0}  \big(\|f_0-\hat{f}_n(\ell)\|_2^2 \big)^2\Big)^{1/2} &= \Big(Var\big(\|f_0-\hat{f}_n(\ell)\|_2^2\big) + \big(E_{f_0}\big(\|f_0-\hat{f}_n(\ell)\|_2^2\big)\big)^2\Big)^{1/2}\nonumber\\
    &\leq \Big(Var\big(\|f_0-\hat{f}_n(\ell)\|_2^2\big)\Big)^{1/2} + E_{f_0}\big(\|f_0-\hat{f}_n(\ell)\|_2^2\big).
\end{align}

We now get
\begin{align}\label{eq:7.5a}
    Var\Big(\|f_0-\hat{f}_n(\ell)\|_2^2\Big)&=Var\Big(\sum_{i=0}^{\ell}\sum_{j=0}^{2^i-1}\big(\frac{1}{m}\sum_{k=1}^{m} y_{ij}^{(k)}-\theta_{0,ij}\big)^2\Big)\nonumber\\
    &=Var\bigg(\sum_{i=0}^{\ell}\sum_{j=0}^{2^i-1}\Big(\frac{1}{m}\sum_{k=1}^m\sqrt{\frac{m}{n}}Z_{ij}^{(k)}\Big)^2\bigg)\nonumber\\
    &=\sum_{i=0}^{\ell}\sum_{j=0}^{2^i-1}Var\bigg(\Big(\frac{1}{m}\sum_{k=1}^m\sqrt{\frac{m}{n}}Z_{ij}^{(k)}\Big)^2\bigg)\nonumber\\
    &=\sum_{i=0}^{\ell}\sum_{j=0}^{2^i-1} (2/n^2) \leq c_7 \frac{2^{\ell}}{n^2},
\end{align}
for some constant $c_7>0$. Now,
\begin{align}\label{eq:8}
    E_{f_0}  \|f_0-\hat{f}_n(\ell)\|_2^2 = E_{f_0}  \|f_0-E(\hat{f}_n(\ell))\|_2^2 + E_{f_0}  \|E(\hat{f}_n(\ell))-\hat{f}_n(\ell)\|_2^2\leq c_8 (2^{-2\ell s} + 2^\ell/n),
\end{align}
for some constant $c_8>0$. Using \eqref{eq:7.5a} and \eqref{eq:8}, we get
\begin{align}\label{eq:7.5b}
    \Big(E_{f_0}  \big(\|f_0-\hat{f}_n(\ell)\|_2^2 \big)^2\Big)^{1/2} \leq \sqrt{c_7}\frac{2^{\ell/2}}{n} + c_8 (2^{-2\ell s} + 2^\ell/n) \leq c_9(2^{-2\ell s} + 2^\ell/n),
\end{align}
for a suitably large constant $c_9>0$. Now, using \eqref{eq:7.5b} and Lemma~\ref{lem: Lep:help} we get 
\begin{align}\label{eq:7}
E_{f_0}\left(\|f_0-\hat{f}_n(\hat{\ell})\|_2^2 \mathds{1}_{\hat{\ell}> \ell^*}\right)&\leq c_9\sum_{\ell=\ell^*+1}^{\ell_{\max}}\left(2^{-2\ell s}+2^\ell/n\right) \exp\left(-\frac{c\tau2^\ell}{2m}\right)\nonumber\\
&= c_9 \sum_{\ell=\ell^*+1}^{\ell_{\max}}2^{-2\ell s} \exp\left(-\frac{c\tau2^\ell}{2m}\right) + c_9\sum_{\ell=\ell^*+1}^{\ell_{\max}}\frac{2^{\ell}}{n} \exp\left(-\frac{c\tau2^\ell}{2m}\right).
\end{align}

Now, the first term of \eqref{eq:7} can be bounded by:
\begin{align}\label{eq:9}
    \sum_{\ell=\ell^*+1}^{\ell_{\max}}2^{-2\ell s} \exp\left(-c\tau\frac{2^\ell}{2m}\right) \leq \sum_{\ell=\ell^*+1}^{\infty}2^{-2\ell s}=2^{-2\ell^*s}\frac{2^{-2s}}{1-2^{-2s}}\leq 2^{-2\ell^*s}\leq c_1^{-2s}n^{-2s/(1+2s)},
\end{align}
and the second term of \eqref{eq:7} can be bounded as follows:
\[
\sum_{\ell = \ell^* + 1}^{\ell_{\max}} \frac{2^\ell}{n} \exp\left( -\frac{c\tau 2^\ell}{2m} \right)
\leq \exp\left( -\frac{c\tau  2^{\ell^*}}{2m} \right)  \frac{2^{\ell_{\max} + 1}}{n}.
\]

Now, as $2^{\ell^*}/m>2^{\ell^*}/2^{\ell_{\min}}=n^{1/(1+2s)-1/(1+2s_{\max})}=n^\delta$, where $\delta>0$, so 
\[\exp\left( -c\tau  2^{\ell^*}/{2m} \right) <\exp(-c\tau n^\delta)<n^{-r}\]
for arbitrary $r>0$ for some large enough $n$. Now, if we set $r=2s/(1+2s)-2s_{\min}/(1+2s_{\min})$, then for some suitably large $n$,  we get
\begin{align}\label{eq:10}
     \exp\left( -\frac{c\tau  2^{\ell^*}}{2m} \right)  \frac{2^{\ell_{\max} + 1}}{n}\leq 2n^{-r}n^{-2s_{\min}/(1+2s_{\min})}\leq 2n^{-2s/(1+2s)}.
\end{align}

Combining \eqref{eq:9} and \eqref{eq:10}, we get 
\begin{align*}
E_{f_0}\left(\|f_0-\hat{f}_n(\hat{\ell})\|_2^2 \mathds{1}_{\hat{\ell}> \ell^*}\right)&\leq c_9\cdot c_1^{-2s}n^{-2s/(1+2s)} + 2c_9n^{-2s/(1+2s)}\\
&=c_{10} n^{-2s/(1+2s)},
\end{align*}
where $c_{10}=c_9(c_1^{-2s}+2)$, finishing the proof of our statement.

\subsection{Proof of Theorem \ref{bits}}\label{bits-proof}

In the first round of communication, we send $T_i^{(k)}$'s for each $i\in\mathcal{L}$, i.e., at most $C\log_2(n)$ (where $C>0$ is a large constant) real numbers from each local machine to the global machine. Now, for each $i\in\mathcal{L}$, we have
\[\mathbb{E}(1\vee T_i^{(k)})\leq 1+\mathbb{E}(T_i^{(k)}) = 1+\frac{\sum_{j=0}^{2^i-1}\theta_{0,ij}^2}{2^i} + \frac{m}{n}\leq 1+\frac{L}{(2^i)^{1+2s}}+\frac{n^{1/(1+2s_{\max})}}{n}. \]

Thus, the conditions of Lemma 2.3 from \cite{Szabo2020AdaptiveConstraints} are satisfied and hence, we get that the maximal number of bits required to communicate each $T_i$ is $C\log_2(n)$ up to an error of order $n^{-c}$, for any $c>0$. Hence, in the first round of communication, the total maximal number of bits transferred is $C_1(\log_2(n))^2$ up to an approximation error of order $\log_2(n)n^{-c}$, where $c>0$ is arbitrarily large.

Now, in the second round of communication, we transfer $y_{ij}^{(k)}$ for $i=0,\ldots,\hat{\ell}, j=0,\ldots,2^i-1$,
\begin{align*}
    \mathbb{E}(1\vee y_{ij}^{(k)})\leq 1+\theta_{0,ij}.
\end{align*}

Thus, at most $2^{\hat{\ell}}$ real numbers are transmitted from the local to the central machine and using Lemma 2.3 from \cite{Szabo2020AdaptiveConstraints} we get that at most $\hat{B}_k=C2^{\hat{\ell}}\log_2(n)$ many bits are transmitted from the $k^{th}$ local to the central machine, where the error of approximation is of the order $n^{-c}$, for some arbitrarily large $c>0$. Now, using Lemma~\ref{lem: Lep:help}, we get
\begin{align*}
    P_{f_0}(\hat{\ell}>\ell^*)\leq \sum_{\ell=\ell^*+1}^{\ell_{\max}}\exp(-c\tau2^\ell/m)\leq c_2\frac{m\log_2(n)}{2^{\ell^*}}\leq c_3\log_2(n)n^{-r},
\end{align*}
where $r=\frac{1}{1+2s}-\frac{1}{1+2s_{\min}}>0$. Now, as these constants are free of the choice of $f_0$, we can obtain the uniform bound
\begin{align}
&\sup_{f_0\in B^{s}_{2,\infty}}P_{f_0}(\hat{\ell}>\ell^*)\leq c_3\log_2(n)n^{-r}.\label{eq:tail}
\end{align}

Finally, since $2^{\ell^*}\leq c'n^{1/(1+2s)}$, it follows from \eqref{eq:tail} that the number of bits transmitted from each local machine satisfies
\[\hat{B}_k\leq C\log_2(n)n^{1/(1+2s)},\]
with high probability uniformly over all$f_0\in B^{s}_{2,\infty}$, completing the proof.

\subsection{Proof of Theorem \ref{thm:Bayes:adapt}}\label{sec:thm:Bayes:adapt}
Let $\varepsilon_{n}=n^{-s/(1+2s)}$ and $\hat{f}^{aggr},\hat{f}^{aggr}(\ell)$ the aggregated empirical Bayes posterior mean and the aggregated posterior mean with fixed thresholding hyper-parameter $\ell$, respectively.  In view of Markov's inequality for every $s>0$

\begin{align}\label{eq:11}
&\sup_{f_0\in B_{2\infty}^{s}(L)}E_{f_0}[\tilde{\Pi}_{\hat{\ell}}(\|f-f_0\|_2\geq M_n\varepsilon_n|Y^{1:m})]\nonumber\\
&\qquad\leq \frac{1}{M_n^2\varepsilon_n^2}\sup\limits_{f_0\in B_{2\infty}^{s}(L)}E_{f_0}\left[E_{\tilde{\Pi}_{\hat{\ell}}}\left(\|f-f_0\|_2^2|Y^{1:m}\right)\right]\nonumber\\
&\qquad\leq \frac{1}{M_n^2\varepsilon_n^2}\sup\limits_{f_0\in B_{2\infty}^{s}(L)}E_{f_0}\Big[E_{\tilde{\Pi}_{\hat{\ell}}}\left(\|f-f_0\|^2_2|Y^{1:m}\right)\cdot\mathds{1}_{\hat{\ell}\leq\ell^*}+ E_{\tilde{\Pi}_{\hat{\ell}}}\left(\|f-f_0\|^2_2|Y^{1:m}\right)\cdot\mathds{1}_{\hat{\ell}>\ell^*}\Big].
%&\leq\frac{1}{M_n^2\varepsilon_n^2}\sup\limits_{f_0\in B_{2\infty}^{s}(L)}E_{f_0}  \int \|f-f_0\|_2^2\tilde{\Pi}_{\hat\ell}(df|Y) \mathds{1}_{\hat{\ell}\leq \ell^*}+ P_{0}(\hat{\ell}>\ell^*),
\end{align}

Now, we first deal with the second term in \eqref{eq:11}. 
\begin{align}\label{eq:11a}
    &E_{f_0}\left[E_{\tilde{\Pi}_{\hat{\ell}}}(\|f-f_0\|^2_2|Y^{1:m})\cdot\mathds{1}_{\hat{\ell}>\ell^*}\right]\nonumber\\
    &=E_{f_0}\left[\left\{E_{\tilde{\Pi}_{\hat{\ell}}}(\|f-E(f|Y^{1:m})\|^2_2|Y^{1:m})+E_{\tilde{\Pi}_{\hat{\ell}}}(\|E(f|Y^{1:m})-f_0\|^2_2|Y^{1:m})\right\}\cdot\mathds{1}_{\hat{\ell}>\ell^*}\right]\nonumber\\
    &=E_{f_0}\left[\left\{\sum_{i=0}^{\hat{\ell}}\sum_{j=0}^{2^i-1}\frac{1}{n+1}+\|E(f|Y^{1:m})-f_0\|^2_2\right\}\cdot\mathds{1}_{\hat{\ell}>\ell^*}\right]\nonumber\\
    &=E_{f_0}\left[\left\{\frac{2^{\hat{\ell}+1}-1}{n+1}+ \sum_{i=0}^{\hat{\ell}}\sum_{j=0}^{2^i-1}\left(\frac{n}{n+1}\bar{y}_{ij}-\theta_{0,ij}\right)^2 + \sum_{i>\hat{\ell}}\sum_{j=0}^{2^i-1}\theta_{0,ij}^2\right\}\cdot\mathds{1}_{\hat{\ell}>\ell^*}\right]\nonumber\\
    &\leq E_{f_0}\left[\left\{\frac{2^{\hat{\ell}+1}-1}{n}+c'2^{-2\hat{\ell}s}\right\}\cdot\mathds{1}_{\hat{\ell}>\ell^*}\right] + E_{f_0}\left[\sum_{i=0}^{\hat{\ell}}\sum_{j=0}^{2^i-1}\left(\frac{n}{n+1}\bar{y}_{ij}-\theta_{0,ij}\right)^2\cdot\mathds{1}_{\hat{\ell}>\ell^*}\right].
\end{align}

Using Lemma~\ref{lem: Lep:help} and adapting the arguments used in the proof of Theorem~\ref{thm: freq:adapt}, we get that the first term in \eqref{eq:11a}
\begin{align}%\label{eq:11b}
    E_{f_0}\left[\left\{\frac{2^{\hat{\ell}+1}-1}{n}+c_12^{-2\hat{\ell}s}\right\}\cdot\mathds{1}_{\hat{\ell}>\ell^*}\right] \leq c_2 n^{-2s/(1+2s)}.\nonumber
\end{align}
We now deal with the second term in \eqref{eq:11a}. It can be written as
\begin{align}%\label{eq:11c}
  & \sum_{\ell=\ell^*+1}^{\ell_{\max}} E_{f_0}\bigg[\sum_{i=0}^{\ell}\sum_{j=0}^{2^i-1}\Big(\frac{n}{n+1}\bar{y}_{ij}-\theta_{0,ij}\Big)^2 \cdot\mathds{1}_{\hat{\ell}=\ell}\bigg]\nonumber\\
  &\leq \sum_{\ell=\ell^*+1}^{\ell_{\max}} \bigg[Var_{f_0}^{1/2} \bigg(\sum_{i=0}^{\ell}\sum_{j=0}^{2^i-1}\Big(\frac{n}{n+1}\bar{y}_{ij}-\theta_{0,ij}\Big)^2\bigg) + E_{f_0}\bigg(\sum_{i=0}^{\ell}\sum_{j=0}^{2^i-1}\Big(\frac{n}{n+1}\bar{y}_{ij}-\theta_{0,ij}\Big)^2\bigg)\bigg]P_{f_0}^{1/2}(\hat{\ell}=\ell)\nonumber\\
  &= \sum_{\ell=\ell^*+1}^{\ell_{\max}} \bigg[\bigg(\sum_{i=0}^{\ell}\sum_{j=0}^{2^i-1}\Big(\frac{2n^2}{(n+1)^4}\Big(1+\frac{2\theta_{0,ij}^2}{n}\Big)\Big)^2\bigg)^{1/2} + \sum_{i=0}^{\ell}\sum_{j=0}^{2^i-1}\bigg(\frac{\theta_{0,ij}^2}{(n+1)^2}+\frac{n}{(n+1)^2}\bigg)\bigg]P_{f_0}^{1/2}(\hat{\ell}=\ell)\nonumber\\
  &\leq \sum_{\ell=\ell^*+1}^{\ell_{\max}} \bigg[c_3\frac{2^{\ell/2}}{n} + \frac{c_4}{n^2}+c_5\frac{2^{\ell}}{n}\bigg]P_{f_0}^{1/2}(\hat{\ell}=\ell)\nonumber\\
  &\leq \sum_{\ell=\ell^*+1}^{\ell_{\max}} c_6\frac{2^{\ell}}{n}P^{1/2}(\hat{\ell}=\ell) + c_7\frac{\log_2 (n)}{n^2}\leq c_8 n^{-2s/(1+2s)},\nonumber
\end{align}
using \eqref{eq:10}. Next, we deal with the first term of \eqref{eq:11}. By elementary computations
\begin{align}\label{eq:12}
 &E_{f_0}\left[E_{\tilde{\Pi}_{\hat{\ell}}}\left(\|f-f_0\|^2_2|Y^{1:m}\right)\mathds{1}_{\hat{\ell}\leq\ell^*}\right]\nonumber\\
 &= E_{f_0}\left(\|\hat{f}^{aggr}-f_{0}\|_2^2\mathds{1}_{\hat{\ell}\leq \ell^*}\right)+E_{f_0}\left(\sum_{i=0}^{\hat{\ell}}\sum_{j=0}^{2^i-1}\frac{1}{n+1}\mathds{1}_{\hat{\ell}\leq \ell^*}\right)\nonumber\\
&\leq E_{f_0}\left(\|\hat{f}_n(\hat{\ell})-f_{0}\|_2^2\mathds{1}_{\hat{\ell}\leq \ell^*}\right)+E_{f_0}\left(\|\hat{f}_n(\hat{\ell})-\hat{f}^{aggr}\|_2^2\mathds{1}_{\hat{\ell}\leq \ell^*}\right)   +c_62^{\ell^*}/n\nonumber\\
&\leq  c_7n^{-2s/(1+2s)}+ E_{f_0}\|\hat{f}_n(\ell^*)-\hat{f}^{aggr}(\ell^*)\|_2^2,
\end{align}
where the last inequality in \eqref{eq:12} follows from the proof of Theorem \ref{thm: freq:adapt}. It remains to deal with the second term of \eqref{eq:12}. Note that
\begin{align*}
E_{f_0}\|f_n(\ell^*)-\hat{f}^{aggr}(\ell^*)\|_2^2&= \frac{1}{(n+1)^2}E_{f_0}\sum_{i=0}^{\ell^*}\sum_{j=0}^{2^i-1} \left(\frac{1}{m} \sum_{k=1}^{m}y_{ij}^{(k)}\right)^2\\
&\leq  \frac{2}{(n+1)^2} \sum_{i=0}^{\ell^*}\sum_{j=0}^{2^i-1} \left( \theta_{0,ij}^2+  E_{f_0}\left(\frac{1}{\sqrt{mn}} \sum_{k=1}^{m}Z_{ij}^{(k)}\right)^2\right)\\
&\leq c_8 {n^{-2s/(1+2s)}}/{n^2},
\end{align*}
finishing the proof of our statement.
\end{noindentsection}

\section{Technical lemmas}
\begin{noindentsection}

\begin{lemma}\label{lem: Lep:help}
For $f\in B_{2,\infty}^s(L)$ we have for every $\ell>\ell^*$ that
\begin{align*}
P_{f_0}(\hat{\ell}=\ell)\leq C e^{-c\tau 2^{\ell}/m},\\
%P_f(\hat{\ell}>\ell^*)\lesssim e^{-c\tau 2^{\ell^*}}
\end{align*}
for some constants $C,c>0$.
\end{lemma}
\begin{proof}
Take any $\ell>\ell^*$ and denote by $\ell^-=\ell-1\geq\ell^*$. If $\hat{\ell}=\ell$, then it is equivalent to say that the following is true:
\[\bigg\{\forall l\geq\ell, \sum_{i=\ell}^{l}\widetilde{T}_i\leq\tau2^l/n\bigg\}\text{ and } \bigg\{\exists l \in \mathcal{L}, l\geq\ell^-\text{ such that }\sum_{i=\ell^-}^{l}\widetilde{T}_i>\tau2^l/n\bigg\}.\]

So,
\begin{align*}
P_{f_0}(\hat{\ell}=\ell) &\leq P_{f_0}\bigg(\exists l \in \mathcal{L},l\geq\ell^-\text{ such that }\sum_{i=\ell^-}^{l}\widetilde{T}_i>\tau2^l/n\bigg) \leq \sum_{l \in \mathcal{L}, l\geq \ell^- }P_{f_0}\Big(\sum_{i=\ell^-}^{l} \widetilde{T}_i>\tau 2^{l}/n\Big).
\end{align*}

Recall that 
\[\widetilde{T}_i=\sum_{j=0}^{2^i-1}\theta_{0,ij}^2+\frac{1}{\sqrt{n}}\sum_{j=0}^{2^i-1}\theta_{0,ij}\frac{1}{\sqrt{m}}\sum_{k=1}^{m}Z_{ij}^{(k)}+ \sum_{j=0}^{2^i-1} \frac{1}{n}\sum_{k=1}^m\Big((Z_{ij}^{(k)})^2-1\Big).\]

Furthermore, note that if $C_s=1/(1-2^{-2s})$, then
\begin{align*}
\sum_{i=\ell^-}^{l}\sum_{j=0}^{2^i-1}\theta_{0,ij}^2\leq B_n(\ell^{-},f)\leq B_n(\ell^*,f)\leq \frac{2^{\ell^*}}{n}\leq  \frac{2^l}{n}.
\end{align*}

Hence,
\begin{align*}
&P_{f_0}\left(\sum_{i=\ell^-}^{l} \widetilde{T}_i>\tau 2^{l}/n\right)\\
&\leq P_{f_0}\left(\sum_{i=\ell^-}^{l} \left\{\frac{1}{\sqrt{mn}}\sum_{j=0}^{2^i-1}\theta_{0,ij}\sum_{k=1}^mZ_{ij}^{(k)}+\sum_{j=0}^{2^i-1} \frac{1}{n}\sum_{k=1}^m\left((Z_{ij}^{(k)})^2-1\right)\right\}>(\tau-1) 2^{l}/n\right)\\
&\leq P_{f_0}\left( \sum_{i=\ell^-}^{l}\sum_{j=0}^{2^i-1} \frac{1}{n}\sum_{k=1}^m((Z_{ij}^{(k)})^2-1)\geq \frac{(\tau-1)2^{l}}{2n} \right)\\
&\quad+ P_{f_0}\left(  \sum_{i=\ell^-}^{l}\sum_{j=0}^{2^i-1}\theta_{0,ij}\sum_{k=1}^mZ_{ij}^{(k)}>\frac{\tau-1}{2}\,\frac{2^l\sqrt{m}}{\sqrt{n}}  \right).
\end{align*}

Furthermore, note that in view of Theorem 3.1.9 of \cite{Gine2015MathematicalModels}, we get for $\tau>2$,
\begin{align*}
P_{f_0}\left( \sum_{i=\ell^-}^{l}\sum_{j=0}^{2^i-1} \frac{1}{n}\sum_{k=1}^m((Z_{ij}^{(k)})^2-1)\geq \frac{(\tau-1)2^{l}}{2n} \right) &\leq\exp\left(-\frac{(\tau-1)^22^{2l}}{4\{\sum_{i=\ell^-}^lm2^i+(\tau-1)2^l/2\}}\right)\\
&\leq\exp\left(-\frac{(\tau-1)^22^{2l}}{4\{m2^{l+1}+(\tau-1)m2^{l+1}\}}\right)\\
&\leq  \exp\{-\frac{(\tau/2)^22^{2l} }{4\tau m 2^{l+1}} \}
\leq \exp\{-d_1 \tau2^\ell/m\},
\end{align*}
for some constant $d_1>0$. Furthermore, using standard upper bound for Gaussian tail probabilities, we get
\begin{align*}
P_{f_0}\left( \Big| \sum_{i=\ell^-}^{l}\sum_{j=0}^{2^i-1}\theta_{0,ij} \sum_{k=1}^mZ_{ij}^{(k)} \Big|>\frac{\tau-1}{2}\,\frac{2^l\sqrt{m}}{\sqrt{n}}  \right)&\leq  \exp\left(-\frac{(\tau-1)^22^{2l}m}{8n\sum_{i=\ell^-}^l\sum_{j=0}^{2^i-1}(\theta_{0,ij}^2m)}\right)\\
&\leq \exp\left(-\frac{(\tau-1)^22^{2l}m}{8n(m2^l/n)}\right)\\
&\leq \exp\left(-\frac{(\tau/2)^22^{2l}m}{8n(m2^l/n)}\right)\leq \exp(-d_2\tau2^\ell) ,
\end{align*}
for some constant $d_2>0$. This concludes the proof.
\end{proof}
\end{noindentsection}

\section{Definition and notation for wavelets}\label{sec:wavelet}
In this section, we introduce the wavelet basis used throughout the paper. A more detailed description can be found in \citet{Gine2015MathematicalModels}. Since our domain is the compact interval $[0,1]$, we require an orthonormal wavelet basis that is properly defined near the boundaries. To this end, we adopt the boundary-corrected construction of \citet{Cohen1993WaveletsTransforms}, known as the Cohen-Daubechies-Vial (CDV) wavelets. The CDV construction modifies the standard compactly supported Daubechies wavelets near the edges of $[0,1]$ to preserve orthonormality while maintaining $N$-regularity and $N$ vanishing moments. For any $N \in \mathbb{N}$, Daubechies' construction yields a father wavelet $\varphi(\cdot)$ and a mother wavelet $\psi(\cdot)$, each with $N$ vanishing moments \citep{Daubechies1992TenWavelets}. Their supports are contained in $[0, 2N-1]$ and $[-N+1, N]$, respectively. Suppose that we translate the father wavelet by $-N+1$ so that it is supported in $[-N+1,N]$ as well, and for notational convenience, we still denote it by $\varphi(\cdot)$. As usual, we define:  $\varphi_{ij}(x)=2^{i/2}\varphi(2^{i}x-k)$, $\psi_{ij}(x)=2^{i/2}\psi(2^{i}x-k)$.

We first construct a orthonormal system in $L^2([0,\infty))$. For a $i_0=i_o(N)$ such that $2^{i_0}\geq 2N$, we define for $j=0,1,\cdots,N-1$:
\[\widetilde{\varphi}_{i_0j}=\sum_{n=j}^{2N-2}\binom{n}{j}2^{i_0/2}\varphi(2^{i_0}x+n-N+1),\;x\geq0.\]

Gram-Schmidt procedure is then applied on them to obtain $\{ \varphi_{i_0j}^\text{left} : j=0,\ldots,N-1\}$ with support contained in $[0,(2N-1)/2^{i_0}]$. Now, the family $\{\varphi_{i_0j}^\text{left},\varphi_{i_0m} : j=0,\ldots,N-1, m\geq N\}$ is an orthonormal system in $L^2\big([0,\infty)\big)$ and generates all polynomials of degree less than or equal to $N-1$. The procedure is repeated on $(-\infty,1]$ to get orthonormal edge basis functions $\{ \varphi_{i_0j'}^\text{right} : j'=-N,\ldots,-1\}$ near the endpoint 1. It is supported in $[1-(2N-1)/2^{i_0},1]$. Thus, put together, the $2^{i_0}-2N$ standard Daubechies father wavelets $\varphi_{i_0m}$ supported in the interior of $[0,1]$, and the $2N$ edge basis functions $\varphi_{i_0j}^\text{left}, \varphi_{i_0j'}^\text{right}$, denoted henceforth by
\[\{\varphi_{i_0j}^{bc}:j=0,\ldots,2^{i_0}-1\}\]
forms an orthonormal system in $L^2([0,1])$, and produces polynomials of degree up to $N-1$ on $[0,1]$.

To construct the corresponding edge-adjusted mother wavelet functions, we first consider $[0,\infty)$, and define:
\[\widetilde{\psi}_{i_0j}=\sqrt{2}\varphi_{i_0j}^\text{left}(2\cdot)-\sum_{m=0}^{N-1}\big\langle\sqrt{2}\varphi_{i_0j}^\text{left}(2\cdot), \varphi_{0m}^\text{left}\big\rangle\,\varphi_{0m}^\text{left},\;j=0,\ldots,N-1,\]
which after Gram-Schmidt orthogonalization yields the orthonormal system:
\[\{\psi_{i_0j}^\text{left},\psi_{i_0m}:j=0,\ldots,N-1,m\geq N\}\]
of $L^2\big([0,\infty)\big)$. Repeating the process over $(-\infty,1]$, we obtain the orthonormal system
\begin{align*}
&\Big\{\psi_{i_0j}^\text{left}, \psi_{i_0j'}^\text{right},\psi_{i_0m}:j=0,\ldots,N-1,j'=-1,\ldots,-N,m=N,\ldots,2^J-N-1\Big\}\\
&\equiv\{\psi_{i_0j}^{bc}:j=0,\ldots,2^{i_0}-1\}
\end{align*}
in $L^2\big([0,1]\big)$. For $i\geq i_0$, we define:
\[\psi_{ij}^\text{left}=2^{(i-i_0)/2}\psi_{i_0j}^\text{left}(2^{i-i_0}\cdot), \psi_{ij}^\text{right}=2^{(i-i_0)/2}\psi_{i_0j}^\text{right}(2^{i-i_0}\cdot)\]

We denote these functions henceforth by: 
\[\{\psi_{lk}^{bc}:k=0,\ldots,2^l-1,l\geq i_0\}.\]

Using these, we construct the wavelet basis functions:
\[
\{ \varphi^{bc}_{i_0k}, \psi^{bc}_{ij} : i > i_0,\; k\in \{0, \ldots, 2^{i_0}-1\},\; j \in \{0, \ldots, 2^i-1\} \}.
\]

For convenience, we define \( \psi^{bc}_{i_0j} := \varphi^{bc}_{i_0j} \), and index the full basis as \( \psi^{bc}_{ij} \) for all \( i \geq i_0 \) and \( j \in \{0, \ldots, 2^i-1\} \). With this basis, any function \( f \in L_2[0,1] \) can be written as:
\[
f = \sum_{i = i_0}^{\infty} \sum_{j=0}^{2^i-1} \theta_{ij} \psi^{bc}_{ij}, \quad \text{with} \quad \theta_{ij} = \langle f, \psi^{bc}_{ij} \rangle.
\]

By the orthonormality of the wavelet basis, the \( L_2 \)-norm of \( f \) satisfies
\[
\|f\|_2^2 = \sum_{i = i_0}^{\infty} \sum_{j=0}^{2^i-1} \theta_{ij}^2.
\]

For simplicity of notation, without loss of generality, we set $i_0=0$ in our paper. 

\end{document}